\documentclass{amsart}

\newcommand{\apref}[3]{\hyperref[#2]{#1\ref*{#2}#3}}

\usepackage{enumerate}
\usepackage[latin1]{inputenc}
\usepackage{dsfont}
\usepackage{amssymb,amsthm,amsmath}

\usepackage{mathrsfs}

\theoremstyle{plain}
\newtheorem{prop}{Proposition}[section]
\newtheorem{lemma}[prop]{Lemma}

\newtheorem{thm}[prop]{Theorem}

\theoremstyle{definition}

\newtheorem{conj}[prop]{Conjecture}

\theoremstyle{remark}
\newtheorem{remark}[prop]{Remark}

\DeclareMathOperator{\Gen}{Gen}
\newcommand{\redu}{\text{red}}

\DeclareMathOperator{\PSL}{PSL}
\DeclareMathOperator{\PGL}{PGL}

\DeclareMathOperator{\Tr}{Tr}

\DeclareMathOperator{\Ima}{Im}
\DeclareMathOperator{\Rea}{Re}

\newcommand{\st}{\text{st}}

\newcommand{\dec}{\text{dec}}

\newcommand\N{\mathbb{N}}

\newcommand\R{\mathbb{R}}
\newcommand\Z{\mathbb{Z}}
\newcommand\C{\mathbb{C}}

\newcommand{\h}{\mathbb{H}}

\newcommand{\mc}[1]{\mathcal #1}

\newcommand{\wt}{\widetilde}
\newcommand{\wh}{\widehat}

\newcommand{\eps}{\varepsilon}

\DeclareMathOperator{\FE}{FE}

\DeclareMathOperator{\id}{id}

\DeclareMathOperator{\Fct}{Fct}

\newcommand{\sceq}{\mathrel{\mathop:}=}

\newcommand{\bmat}[4]{\begin{bmatrix} #1&#2\\#3&#4\end{bmatrix}}

\newcommand{\textbmat}[4]{\left[\begin{smallmatrix} #1&#2 \\ #3&#4
\end{smallmatrix}\right]}

\usepackage[colorlinks,breaklinks]{hyperref}
\usepackage{graphics}
\usepackage[T1]{fontenc} 
\usepackage{pdflscape}
\allowdisplaybreaks

\begin{document}

\title[Odd and even Maass cusp forms]{Odd and even Maass cusp forms for Hecke triangle groups, and the billiard flow}
\author[A.\@ Pohl]{Anke D.\@ Pohl}
\address{Mathematisches Institut, Georg-August-Universit\"at G\"ottingen,  Bunsenstr. 3-5, 37073 G\"ottingen}
\email{pohl@uni-math.gwdg.de}
\begin{abstract} 
By a transfer operator approach to Maass cusp forms and the Selberg zeta function for cofinite Hecke triangle groups, M.\@ M\"oller and the author found a factorization of the Selberg zeta function into a product of Fredholm determinants of transfer-operator-like families:
\[
 Z(s) = \det(1-\mc L_s^+)\det(1-\mc L_s^-).
\]
In this article we show that the operator families $\mc L_s^\pm$ arise as families of transfer operators for the triangle groups underlying the Hecke triangle groups, and that for $s\in\C$, $\Rea s=\tfrac12$, the operator $\mc L_s^+$ (resp.\@ $\mc L_s^-$) has a $1$-eigenfunction if and only if there exists an even (resp.\@ odd) Maass cusp form with eigenvalue $s(1-s)$. For nonarithmetic Hecke triangle groups, this result provides a new formulation of the Phillips-Sarnak conjecture on nonexistence of even Maass cusp forms.
\end{abstract}
\thanks{The author acknowledges the support by the Volkswagen Foundation}
\maketitle


\section{Introduction and statement of main results}

This article is part of a program to study the connections between Maass cusp forms for Fuchsian groups and eigenfunctions of transfer operators derived from discretizations of the geodesic flow on the associated hyperbolic Riemannian surfaces and orbifolds \cite{Pohl_diss, Hilgert_Pohl, Pohl_Symdyn2d, Moeller_Pohl, Pohl_mcf_Gamma0p, Pohl_mcf_general,  Mayer_selberg, Mayer_thermo, Chang_Mayer_eigen, Chang_Mayer_extension, Efrat_spectral, Mayer_Muehlenbruch_Stroemberg, Morita_transfer, Pollicott, Deitmar_Hilgert,  Lewis,   Lewis_Zagier, BLZ_part2, Pohl_hecke_infinite}. Here, we consider the cofinite Hecke triangle groups
\[
 \Gamma_q \sceq \left\langle \bmat{0}{1}{-1}{0}, \bmat{1}{2\cos\frac{\pi}{q}}{0}{1}\right\rangle \quad\leq \PSL_2(\R)\qquad (q\in\N_{\geq 3}).
\]
They form a mixed family of arithmetic and nonarithmetic nonuniform Fuchsian lattices with $\Gamma_3$ being the modular group $\PSL_2(\Z)$ (see Section~\ref{prelims} below for more details). Up to date for none of these Hecke triangle groups the Fourier coefficients of a single Maass cusp form could be provided as explicit numbers. However, for each Hecke triangle group it is known that the Weyl law holds for its odd Maass cusp forms, and for $\Gamma_3$ also for its even Maass cusp forms (see \cite{Roelcke, Selberg_Goe, Venkov_book}). In strict contrast, for nonarithmetic Hecke triangle groups (that is for $\Gamma_q$ other than $\Gamma_3,\Gamma_4,\Gamma_6$) the Phillips-Sarnak conjecture \cite{Phillips_Sarnak_cuspforms, Phillips_Sarnak_weyl, Judge} states that even Maass cusp forms should not exist.

In \cite{Moeller_Pohl}, M\"oller and the author provide transfer operator approaches to Maass cusp forms and the Selberg zeta function for Hecke triangle groups. The transfer operator families arise from the discretization and symbolic dynamics in \cite{Pohl_Symdyn2d} for the geodesic flow on $\Gamma_q\backslash\h$, with $\h$ being the hyperbolic plane, as well as from a certain acceleration of this discretization. The main results in \cite{Moeller_Pohl} are (see Section~\ref{sec_known} below for notation and more details):
\begin{itemize}
\item The space of Maass cusp forms for $\Gamma_q$ with eigenvalue $s(1-s)$ with $s\in\C, \Rea s \in (0,1)$, is isomorphic to the space of $1$-eigenfunctions of certain regularity of the finite-term transfer operator 
\[
 \mc L_{F,s} = \sum_{k=1}^{q-1}\tau_s(g_k).
\]
Here $g_1,\ldots, g_{q-1}$ are certain elements in $\Gamma_q$ and $\tau_s$ denotes the action of the principal series representation with spectral parameter $s$ in the line model. 

We remark that for Hecke triangle groups, the first eigenvalue is strictly larger than $\tfrac14$ (\cite[p.\@ 583(8)]{Hejhal2}). In particular, for each such eigenvalue $s(1-s)$ one has $\Rea s = \tfrac12$. The statement that the $1$-eigenfunction spaces of $\mc L_{F,s}$ are also isomorphic to the Maass cusp form spaces with eigenvalue $s(1-s)$ if $\Rea s\in (0,1)$, $\Rea s \not=\tfrac12$, means that $\mc L_{F,s}$ does not have any such eigenfunctions of the required regularity. 
\item The Selberg zeta function $Z$ equals the Fredholm determinant of the infinite-term transfer operator family $\mc L_{G,s}$ from the accelerated discretization:
\[
 Z(s) = \det(1-\mc L_{G,s}),
\]
where
\[
 \mc L_{G,s} = 
\begin{pmatrix}
0 & \sum\limits_{k=2}^{q-2} \tau_s(g_k) & \sum\limits_{n\in\N}\tau_s(g_1^n)
\\
\sum\limits_{n\in\N} \tau_s(g_{q-1}^n) & \sum\limits_{k=2}^{q-2}\tau_s(g_k) & \sum\limits_{n\in\N}\tau_s(g_1^n)
\\
\sum\limits_{n\in\N} \tau_s(g_{q-1}^n) & \sum\limits_{k=2}^{q-2}\tau_s(g_k) & 0
\end{pmatrix}
\]
is defined on a certain Banach space of holomorphic functions with continuous extensions. Thus, the zeros of $Z$ are determined by the $1$-eigenfunctions of $\mc L_{G,s}$ (and its meromorphic continuation). In particular, Selberg theory implies that the $1$-eigenfunctions of $\mc L_{G,s}$ for $s\in\C$, $\Rea s=\tfrac12$, are related to Maass cusp forms.
\item The symbolic dynamics are in a certain sense compatible with the orientation-reversing Riemannian isometry
\[
 Q \colon z\mapsto \frac{1}{\overline{z}}.
\]
The eigenfunctions of the transfer operators $\mc L_{F,s}$ and $\mc L_{G,s}$ can be split into $\tau_s(Q)$-invariant and $\tau_s(Q)$-anti-invariant parts. The transfer operators themselves commute with the action of $Q$. This allows to find transfer-operator-like families $\mc L_{F,s}^\pm$ and $\mc L_{G,s}^\pm$ such that the $\tau_s(Q)$-invariant (resp.\@ the $\tau_s(Q)$-anti-invariant) eigenfunctions of $\mc L_{F,s}$ are characterized as eigenfunctions of $\mc L_{F,s}^+$ (resp.\@ of $\mc L_{F,s}^-$), and analogously for $\mc L_{G,s}$.
\item The spaces of odd resp.\@ even Maass cusp forms for $\Gamma_q$ are isomorphic to the spaces of $1$-eigenfunctions of $\mc L_{F,s}^-$ resp.\@ of $\mc L_{F,s}^+$.
\item The Selberg zeta function factorizes into a product of two Fredholm determinants
\[
 Z(s) = \det(1-\mc L_{G,s}^+) \det(1-\mc L_{G,s}^-).
\]
\end{itemize}

These results provide a characterization of Maass cusp forms by purely classical dynamical entities as well as a formulation of the Phillips-Sarnak conjecture in terms of nonexistence of $1$-eigenfunctions of the transfer operator $\mc L_{F,s}^+$. Moreover, they lead to the following conjecture.

\begin{conj}[\cite{Moeller_Pohl}]\label{conj}
The spaces of $1$-eigenfunctions of $\mc L_{F,s}^+$ and $\mc L_{G,s}^+$ resp.\@ of $\mc L_{F,s}^-$ and $\mc L_{G,s}^-$ are isomorphic.
\end{conj}

For the modular group $\Gamma_3$, the transfer operators $\mc L_{F,s}^\pm$ determine the functional equation of odd resp.\@ even period functions from \cite{Lewis_Zagier} and $\mc L_{G,s}^\pm$ is Mayer's transfer operator resp.\@ its negative \cite{Mayer_selberg}. In this case, Conjecture~\ref{conj} has been established first on a spectral level by Efrat \cite{Efrat_spectral}, who showed that $\det(1-\mc L_{G,s}^+)$ (resp.\@ $\det(1-\mc L_{G,s}^-)$) has a zero for $s\in\C$, $\Rea s >0$, if and only if $s$ is a parameter of the even spectrum (resp.\@ of the odd spectrum) in $L^2(\Gamma_q\backslash\h)$. For the proof he used continued fractions and the reduction theory of indefinite binary quadratic forms. Later, the full statement of Conjecture~\ref{conj} was proven for $\Gamma_3$ by \cite{Chang_Mayer_transop, Lewis_Zagier}.

In this article we prove a spectral version of Conjecture~\ref{conj} for all Hecke triangle groups.

\begin{thm}\label{mainintro}
For $s\in\C$, $\Rea s > 0$, the Fredholm determinant $\det(1-\mc L_{G,s}^-)$ has a zero if and only if $s$ is the spectral parameter of an odd Maass cusp form. The Fredholm determinant $\det(1-\mc L_{G,s}^+)$ has a zero if and only if $s$ is a parameter of the even spectrum.
\end{thm}

For each Hecke triangle group $\Gamma_q$ there exists a unique triangle group $\wt\Gamma_q$ in $\PGL_2(\R)$, namely $\wt\Gamma_q = \langle \Gamma_q, Q\rangle$, in which $\Gamma_q$ has index $2$. For the proof of Theorem~\ref{mainintro} we will develop a transfer operator approach (thermodynamic formalism) to Selberg-type zeta functions for the (geodesic) billiard flow on $\wt\Gamma_q\backslash\h$. As a by-product, we reprove Efrat's result with dynamical methods. Moreover, we achieve a new formulation of the Phillips-Sarnak conjecture saying that for nonarithmetic Hecke triangle groups, the transfer operators $\mc L_{G,s}^+$  with $\Rea s = \tfrac12$ should not have $1$-eigenfunctions in a certain Banach space. 

In Sections~\ref{prelims} and \ref{sec_known} we recall the necessary background knowledge on Hecke triangle groups, the Selberg zeta function and Maass cusp forms as well as the results from the transfer operator approaches in \cite{Moeller_Pohl}. In Section~\ref{sec_strategy} we briefly present the strategy of the proof of Theorem~\ref{mainintro}, which we perform in Section~\ref{odddone} for odd $q$ and in Section~\ref{evendone} for even $q$. We conclude with a few remarks in Section~\ref{conclusion}.

Throughout we use $\N = \{1,2,3,\ldots\}$ and $\N_0 = \N \cup \{0\}$.

\section{Preliminaries on Hecke triangle groups, the Selberg zeta function and Maass cusp forms}\label{prelims}

For $q\in\N$, $q\geq 3$, the \textit{Hecke triangle group} $\Gamma_q$ is the Fuchsian lattice in $\PSL_2(\R)$ which is generated by the two elements
\[
 S\sceq \bmat{0}{1}{-1}{0} \quad\text{and}\quad T_q\sceq \bmat{1}{\lambda_q}{0}{1}
\]
in $\PSL_2(\R)$, where $\lambda_q \sceq 2\cos\frac{\pi}{q}$. As well-known, $\Gamma_q$ acts on the upper half-plane model 
\[
 \h \sceq \{z\in\C \mid \Ima z>0\}
\]
of the hyperbolic plane by fractional linear transformations:
\[
 \bmat{a}{b}{c}{d}.z  = \frac{az+b}{cz+d}.
\]
A fundamental domain for $\Gamma_q$ is given by e.g.\@ the Ford fundamental domain
\[
 \mc F_q \sceq \left\{ z\in\h \left\vert\ |\Rea z|<\frac{\lambda_q}2,\ |z|>1\right.\right\},
\]
see Figure~\ref{Forddom}. The orbifold $\Gamma_q\backslash\h$ has one cusp, represented by $\infty$, and two elliptic points, represented by $i$ and 
\[
 \varrho_q \sceq \frac{\lambda_q + i\sqrt{4-\lambda_q^2}}2.
\]
\begin{figure}[h]
\begin{center}
\includegraphics*{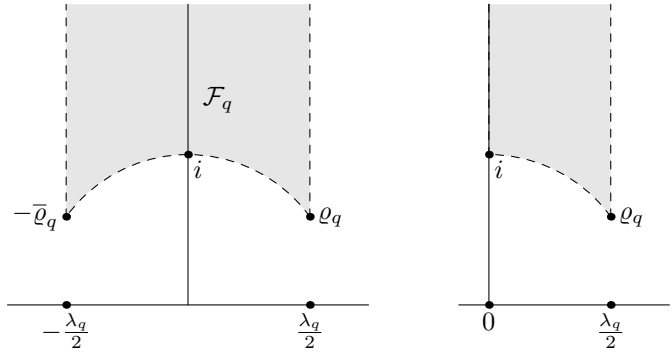} 
\end{center}
\caption{On the left side: a fundamental domain for $\Gamma_q$ in $\h$; on the right side: a fundamental domain for $\wt\Gamma_q$ in $\h$}\label{Forddom}
\end{figure}

The order of $i$ is $2$ and its stabilizer group is $\{\id, S\}$. The order of $\varrho_q$ is $q$ with stabilizer group $\{ \id, U_q^1, U_q^2, \ldots, U_q^{q-1}\}$, where 
\[
 U_q \sceq T_qS = \bmat{\lambda_q}{-1}{1}{0}.
\]
For $q=3$, the Hecke triangle group $\Gamma_3$ is the well-investigated modular group $\PSL_2(\Z)$, and $\Gamma_3\backslash\h$ is the modular surface. The lattices $\Gamma_3,\Gamma_4$ and $\Gamma_6$ are arithmetic, all other $\Gamma_q$ are nonarithmetic.

A \textit{Maass cusp form} for $\Gamma_q$ is a $C^\infty$-function $u\colon \h \to \C$ which 
\begin{itemize}
\item is an eigenfunction of the hyperbolic Laplace-Beltrami operator 
\[
 \Delta = -y^2\left(\frac{\partial^2}{\partial x^2} + \frac{\partial^2}{\partial y^2}\right),
\]
\item is constant on $\Gamma_q$-orbits, that is for all $g\in\Gamma_q$ and $z\in\h$ we have $u(g.z) = u(z)$, and factors to an element of $L^2(\Gamma_q\backslash\h)$, and
\item for which we have
\[
 \int_0^{\lambda_q} u(x+iy) dx = 0
\]
for Lebesgue-almost all $y>0$.
\end{itemize}
Because for Hecke triangle groups the residual spectrum is trivial, the Hilbert space $L^2(\Gamma_q\backslash\h)$ decomposes as
\[
 L^2(\Gamma_q\backslash\h) = \C \oplus L^2_{\text{cusp}}(\Gamma_q\backslash\h) \oplus L^2_{\text{cont}}(\Gamma_q\backslash\h),
\]
where the cuspidal part $L^2_{\text{cusp}}(\Gamma_q\backslash\h)$ is spanned by the Maass cusp forms, and the continuous part $L^2_{\text{cont}}(\Gamma_q\backslash\h)$ is determined by Eisenstein series.

A function $f\in L^2(\Gamma_q\backslash\h)$ is called \textit{even} if $f(z) = f(-\overline{z})$, and it is called \textit{odd} if $f(z) = -f(-\overline{z})$. The isometry $z\mapsto -\overline{z}$ commutes with the Laplace-Beltrami operator $\Delta$. We have the decomposition
\[
 L^2(\Gamma_q\backslash\h) = L^2_{\text{even}}(\Gamma_q\backslash\h) \oplus L^2_{\text{odd}}(\Gamma_q\backslash\h).
\]
The restriction of $\Delta$ to $L^2_{\text{odd}}(\Gamma_q\backslash\h)$ is of purely discrete spectrum, and the continuous spectrum of $\Delta$ belongs to $L^2_{\text{even}}(\Gamma_q\backslash\h)$.

Selberg theory (see e.g.\@ \cite{Venkov_book}) shows that various spectral information is encoded in the \textit{Selberg zeta function} 
\[
 Z(s) = \prod_{\wh\gamma}\prod_{k=0}^\infty\left( 1-e^{-(s+k)\ell(\wh\gamma)}\right)
\]
for $\Gamma_q$. Here, the outer product runs over all primitive periodic geodesics $\wh\gamma$ on $\Gamma_q\backslash\h$, and $\ell(\wh\gamma)$ denotes the length of $\wh\gamma$. The zeta function $Z$ converges absolutely for $\Rea s>1$ and extends meromorphically to all of $\C$. We denote its meromorphic extension also by $Z$.

We define
\[
 Q\sceq \bmat{0}{1}{1}{0} \quad\in\PGL_2(\R),
\]
which is identified with the orientation-reversing Riemannian isometry $z\mapsto 1/\overline{z}$. Let $\wt\Gamma_q$ denote the triangle group which is given by the extension of $\Gamma_q$ with $Q$. A fundamental domain for $\wt\Gamma_q$ in $\h$ is indicated in Figure~\ref{Forddom}.

In \cite{Venkov_book}, Venkov shows that the restriction of $\Delta$ to $L^2_{\text{odd}}(\Gamma_q\backslash\h)$ is isomorphic to the Laplace operator of the Dirichlet boundary value problem on $\wt\Gamma_q\backslash\h$ with vanishing boundary values, as well as that the restriction of $\Delta$ to $L^2_{\text{even}}(\Gamma_q\backslash\h)$ is isomorphic to the Neumann boundary value problem on $\wt\Gamma_q\backslash\h$ with vanishing exterior normal derivative. Modifying the kernel of the Selberg trace formula accordingly, he derives a Selberg-type zeta function $Z^V_-$ for the Dirichlet boundary value problem, and a Selberg-type zeta function $Z^V_+$ for the Neumann boundary value problem. He provides an explicit expression for $Z^V_-$ in \cite[(6.6.2)]{Venkov_book}, and the expression for $Z^V_+$ is then easily derived by $Z^V_+ \sceq Z^4/Z^V_-$ (cf.\@ Theorem~6.6.3 and the remark after Theorem~6.5.5 in \cite{Venkov_book}). The expressions for $Z^V_\pm$ depend on whether $q$ is even or odd, for which reason 
we only state them in Sections~\ref{Venkov_odd} and \ref{Venkov_even} below. These zeta functions converge absolutely for $\Rea s >1$ and extend to meromorphic functions on all of $\C$. We also use $Z^V_\pm$ to denote their meromorphic continuations. 
Venkov characterizes their zeros and poles in $\C$. They have trivial zeros at $s=-k$, $k\in\N_0$. A nontrivial zero of $Z^V_+$ we call a parameter of the even spectrum. For the nontrivial zeros, we state Venkov's result mainly restricted to the cuspidal part of the spectrum. We note that for Hecke triangle groups it is known that the first eigenvalue is strictly larger than $\tfrac14$ (\cite[p.\@ 583(8)]{Hejhal2}).

\begin{thm}[Theorems~5.1.3, 5.1.4, 6.6.3 in \cite{Venkov_book}]\label{Venkov}
Let $s\in\C$, $\Rea s > 0$. 
\begin{enumerate}[{\rm (i)}]
\item Then $s$ is a zero of $Z^V_-$ if and only if $s(1-s)$ is an eigenvalue of an odd Maass cusp form. The order of $s$ as a zero equals four times the dimension of the space of odd Maass cusp forms with eigenvalue $s(1-s)$.
\item If $\Rea s = \tfrac12$, then $s$ is a zero of $Z^V_+$ if and only if $s(1-s)$ is an eigenvalue of an even Maass cusp form. The order of $s$ as a zero equals four times the dimension of the space of even Maass cusp forms with eigenvalue $s(1-s)$.
\item If $\Rea s \not= \tfrac12$, then $s$ is a zero of $Z^V_+$ (of order $4n$) if and only if $s$ is a zero of $Z$ (of order $n$). These zeros do not correspond to eigenvalues of Maass cusp forms.
\end{enumerate}
\end{thm}


\section{The geodesic flow, transfer operators, and Maass cusp forms}\label{sec_known}

We recall the symbolic dynamics from \cite{Pohl_Symdyn2d,Moeller_Pohl} for the geodesic flow on $\Gamma_q\backslash\h$ and the results from the transfer operator approaches to Maass cusp forms and the Selberg zeta function from \cite{Moeller_Pohl}. These transfer operator approaches are based on discrete dynamical systems which arise from specific cross sections for the geodesic flow on $\Gamma_q\backslash\h$. 

We denote by $S\h$ the unit tangent bundle of $\h$, and identify the unit tangent bundle of $\Gamma_q\backslash\h$ with $\Gamma_q\backslash S\h$. For any unit tangent vector $\wh v \in \Gamma_q\backslash S\h$ we let $\wh \gamma_{\wh v}$ denote the geodesic on $\Gamma_q\backslash\h$ determined by
\[
 \wh\gamma_{\wh v}'(0) = \wh v.
\]
We say that $\wh\gamma_{\wh v}$ intersects a given subset of $\Gamma_q\backslash S\h$ if for some $t\in\R$, the vector $\wh\gamma_{\wh v}'(t)$ is contained in this subset. A \textit{cross section} for the geodesic flow on $\Gamma_q\backslash\h$ refers here to a subset $\wh C$ of $\Gamma_q\backslash S\h$ which is intersected by each \textit{periodic} geodesic on $\Gamma_q\backslash\h$ and for which each intersection between a geodesic and $\wh C$ is discrete in time. For all cross sections we consider here, the \textit{first return map} $R\colon \wh C \to \wh C$,
\[
 R(\wh v) \sceq \wh\gamma_{\wh v}'(t_0),
\]
where 
\[
 t_0 \sceq \min\left\{ t>0 \left\vert\ \wh\gamma_{\wh v}'(t) \in \wh C\right.\right\}
\]
is the \textit{first return time}, is well-defined. 

For the cross sections from \cite{Pohl_Symdyn2d, Moeller_Pohl} one finds discrete dynamical systems on parts of $\R$ to which the first return map is semi-conjugate. We now describe the general construction of these discrete dynamical systems. Let $\pi\colon S\h\to \Gamma_q\backslash S\h$ denote the quotient map. A \textit{set of representatives} for a cross section $\wh C$ is a subset $C'$ of $S\h$ such that $\pi\vert_{C'}$ provides a bijection between $C'$ and $\wh C$. For $v\in S\h$ we let $\gamma_v$ denote the geodesic on $\h$ determined by $\gamma'_v(0) = v$. For each cross section $\wh C$ from \cite{Pohl_Symdyn2d, Moeller_Pohl} there exists a set of representatives $C'$ such that the map
\begin{equation}\label{tau}
 \tau\colon \wh C \to \R,\quad \wh v \mapsto \gamma_v(\infty),
\end{equation}
where $v\sceq \big(\pi\vert_{C'}\big)^{-1}(\wh v)$, is injective. We remark that $\tau$ depends on the choice of $C'$. Moreover,  the map $\tau$ can be turned into an intertwiner, but the information from the past of the geodesics will not be needed.

A point in the geodesic boundary $\overline \h \cong P^1(\R) \cong \R \cup \{\infty\}$ of $\h$ is called \textit{cuspidal} (for $\Gamma_q$) if it is a representative of the cusp. Here, the set of cuspidal points is $\Gamma_q.\infty$. For any subset $I$ of $\R$, we set
\[
 I_\st \sceq I \setminus \Gamma_q.\infty.
\]
The subscript ``$\st$'' stands for ``strong''. It refers to strong cross sections as opposed to weak cross sections, a distinction from \cite{Pohl_Symdyn2d}. All cross sections $\wh C$ we consider here are strong ones. Moreover, there exists an (open) interval $I$ in $\R$ such that 
\[
 \tau(\wh C) = I_\st.
\]
Then we find a map $H\colon I_\st \to I_\st$ such that $\tau\circ R = H \circ\tau$, that is, $(\wh C, R)$ is semi-conjugate to $(I_\st, H)$. The map $H$ decomposes into maps of the form
\begin{equation}\label{submap}
 I_{g,\st} \to H(I_{g,\st}),\quad x\mapsto g.x,
\end{equation}
for some $g\in \Gamma_q$ and $I_g$ a subinterval of $I$. If $I_g$ is chosen maximal, we call the map in \eqref{submap} a \textit{submap} of $H$. The map $H$ may decompose into finitely or infinitely many submaps. 

The discrete dynamical system $(I_\st, H)$ induces a family of transfer operators $\mc L_{H,s}$, $s\in \C$, by
\[
 \big( \mc L_{H,s} f\big)(x) \sceq \sum_{y\in H^{-1}(x)} \frac{f(y)}{|H'(y)|^s},
\]
where $f$ belongs to an appropriate space of functions on $I$ (more precisely on $I_\st$ and then extended to $I$). The precise space of functions depends on $H$ and applications. It will be discussed further below.

In the following subsections we recall the definitions of the two cross sections from \cite{Pohl_Symdyn2d, Moeller_Pohl} for the geodesic flow on $\Gamma_q\backslash \h$, which both give rise to discrete dynamical systems on $D_\st \sceq (0,\infty)_\st$. The so-called \textit{slow} one from \cite{Pohl_Symdyn2d} induces a discrete dynamical system, denoted by $(D_\st, F)$, which decomposes into finitely many submaps. The term ``slow'' refers to the property that the iterated action may involve consecutive applications of the same parabolic element from $\Gamma_q$. The $1$-eigenfunctions of the arising transfer operators are in bijection with Maass cusp forms. In contrast, the so-called \textit{fast} one from \cite{Moeller_Pohl} induces a discrete dynamical system, denoted by $(D_\st, G)$, which does not allow consecutive applications of the same parabolic element, but decomposes into infinitely many submaps. The arising family of transfer operators consists of nuclear operators of order $0$ on a certain 
Banach 
space. The Fredholm determinant of this family equals the Selberg zeta function. The cross sections, discrete 
dynamical systems and transfer operators depend on the specific Hecke triangle group $\Gamma_q$. However, to avoid overly decorated symbols, we omit references to $\Gamma_q$ from the notation. Recall that 
\[
 S = \bmat{0}{1}{-1}{0},\quad T = \bmat{1}{\lambda}{0}{1}\quad\text{and}\quad  U = TS = \bmat{\lambda}{-1}{1}{0}
\]
with $\lambda\sceq \lambda_q = 2\cos \tfrac{\pi}{q}$. For $k\in \Z$ we define
\[
 g_k \sceq \big( U^k S \big)^{-1}.
\]
Recall that $U^q = \id$. For $s\in\C$, $g=\textbmat{a}{b}{c}{d}\in \PGL_2(\R)$ and $t\in\R$, we set
\[
 j_s(g,t) \sceq \big( (ct+d)^{-2} \big)^s.
\]
If, in addition, $f\colon V\to \R$ is a function on some subset $V$ of $\R$ (or $P^1(\R)$), then we define 
\[
 \tau_s(g^{-1})f(t) \sceq j_s(g,t) f(g.t)
\]
whenever this makes sense. Recall that $Q= \textbmat{0}{1}{1}{0}$.

\subsection{Slow discrete dynamical system, transfer operators, and Maass cusp forms}\label{sec_MPslow}

The slow discrete dynamical system $(D_\st, F)$ from \cite{Pohl_Symdyn2d} is induced by the cross section $\wh C_F$ with
\[
 C'_F = \{ v\in S\h \mid\text{$v$ is based on $i\R_{>0}$, $\gamma_v(\infty) \in (0,\infty)_\st$}\}
\]
as set of representatives. The set $C'_F$ consists of the unit tangent vectors with base point on the imaginary axis $i\R_{>0}$ which point to the right, and such that the geodesics determined by these vectors do not converge to a cuspidal point. The map $F\colon D_\st \to D_\st$ is given by the submaps
\[
 (g_k^{-1}.0, g_k^{-1}.\infty)_\st \to D_\st,\quad x \mapsto g_k.x,
\]
for $k=1,\ldots, q-1$. The associated family of transfer operators
\[
 \mc L_{F,s} = \sum_{k=1}^{q-1} \tau_s(g_k)
\]
is defined on the space of complex-valued functions on $(0,\infty)$. 

For $s\in\C$ let $C^\omega_{\dec,s}(\R_{>0};\C)$ denote the space of real-analytic functions $f\colon\R_{>0}\to\C$ such that the map
\[
x\mapsto 
\begin{cases}
f(x) & \text{for $x\in \R_{>0}$,}
\\
-\tau_s(S)f(x) & \text{for $x\in \R_{<0}$}
\end{cases}
\]
extends smoothly to $\R$. Let 
\[
 \FE_s(\R_{>0})^\dec_\omega \sceq \{ f\in C^\omega_{\dec,s}(\R_{>0};\C) \mid f=\mc L_{F,s}f\}
\]
denote the subspace of $1$-eigenfunctions of $\mc L_{F,s}$, and let 
\begin{align*}
\FE_s(\R_{>0})^{\dec,+}_{\omega} &\sceq \{ f\in \FE_s(\R_{>0})^\dec_\omega \mid \tau_s(Q)f = f\}
\intertext{respectively}
\FE_s(\R_{>0})^{\dec,-}_{\omega} &\sceq \{ f\in \FE_s(\R_{>0})^\dec_\omega \mid -\tau_s(Q)f = f\}
\end{align*}
denote the subspace of $\tau_s(Q)$-invariant resp.\@ $\tau_s(Q)$-anti-invariant functions in the space $\FE_s(\R_{>0})^\dec_\omega$.

\begin{thm}[\cite{Moeller_Pohl}]\label{periodfunctions}
Let $s\in\C$, $\Rea s \in (0,1)$. Then the space $\FE_s(\R_{>0})^{\dec}_\omega$ is isomorphic as a vector space to the space of Maass cusp forms for $\Gamma_q$ with eigenvalue $s(1-s)$. More precisely, $\FE_s(\R_{>0})^\dec_\omega$ decomposes into the direct sum
\[
 \FE_s(\R_{>0})^\dec_\omega = \FE_s(\R_{>0})^{\dec,+}_\omega \oplus \FE_s(\R_{>0})^{\dec,-}_\omega,
\]
and $\FE_s(\R_{>0})^{\dec,+}$ resp.\@ $\FE_s(\R_{>0})^{\dec,-}_\omega$ is isomorphic to the space of even resp.\@ odd Maass cusp forms for $\Gamma_q$ with eigenvalue $s(1-s)$.
\end{thm}

The identity $f=\mc L_{F,s}f$ in the definition of $\FE_s(\R_{>0})^\dec_\omega$ is equal to the functional equation 
\[
 f = \sum_{k=1}^{q-1} \tau_s(g_k)f.
\]
Moreover, the isomorphism in Theorem~\ref{periodfunctions} is given by an integral transform. For these reasons we call the elements in $\FE_s(\R_{>0})^\dec_\omega$ \textit{period functions} for $\Gamma_q$. Accordingly, the elements of $\FE_s(\R_>)^{\dec,+}_\omega$ (resp.\@ of $\FE_s(\R_{>0})^{\dec,-}_\omega$) are called \textit{even} (resp.\@ \textit{odd}) period functions. As shown in \cite{Moeller_Pohl}, also even and odd period functions for $\Gamma_q$ can be characterized as $C^\omega_{\dec, s}$-functions on $\R_{>0}$ satisfying a single functional equation. For that let
\[
 m\sceq \left\lfloor \frac{q+1}{2} \right\rfloor.
\]
Then we have
\[
 \FE_s(\R_{>0})^{\dec,\pm}_\omega = \Big\{ f\in C^\omega_{\dec,s}(\R_{>0};\C) \ \Big\vert\ f= \sum_{k=m}^{q-1} \tau_s(g_k)f \pm \tau_s(Qg_k)f \Big\}
\]
for $q$ odd, and
\begin{align*}
 \FE_s(\R_{>0})^{\dec,\pm}_\omega = \Big\{ &f\in C^\omega_{\dec,s}(\R_{>0};\C) \ \Big\vert\ 
\\
& f=\sum_{k=m+1}^{q-1} \tau_s(g_k)f \pm \tau_s(Qg_k)f + \frac12\tau_s(g_m)f \pm \frac12\tau_s(Qg_m)f\Big\}
\end{align*}
for $q$ even.

\subsection{Fast discrete dynamical system, transfer operators, and the Selberg zeta function}

The fast discrete dynamical system $(D_\st, G)$ from \cite{Moeller_Pohl} is an acceleration of the slow discrete dynamical system $(D_\st, F)$ on its two parabolic elements 
\[
 g_1 = \bmat{1}{-\lambda}{0}{1} \quad\text{and}\quad g_{q-1}=\bmat{1}{0}{-\lambda}{1}.
\]
Also $(D_\st, G)$ arises from a cross section, which we recall in the following. For that let $R_F$ denote the first return map associated to the cross section $\wh C_F$ inducing the system $(D_\st, F)$, and let $\tau_F$ be the map associated to $C'_F$ as in  \eqref{tau}. Define
\[
 \wh{\mc N}_G \sceq \big\{ \wh v\in \wh C_F\ \big\vert\  \tau_F(R_F^{-1}(\wh v)) \in (g_1^{-2}.0,\infty)_\st \cup (0,g_{q-1}^{-2}.\infty)_\st \big\}.
\]
These are the vectors in $\wh C_F$ which correspond to multiple consecutive applications of $g_1$ or $g_{q-1}$ in the iteration of $F$. The cross section we consider in the following is 
\[
 \wh C_G \sceq \wh C_F\setminus \wh{\mc N}_G.
\]
As set of representatives we use
\[
 C'_G \sceq C'_F \setminus \mc N_G,
\]
where $\mc N_G \sceq \pi^{-1}(\wh{\mc N}_G)$. The map $G\colon D_\st \to D_\st$ is then given by the submaps
\[
 \big(g_k^{-1}.0,g_k^{-1}.\infty\big)_\st \to D_\st,\quad x\mapsto g_k.x,
\]
for $k=2,\ldots, q-2$, and
\[
 \big(g_1^{-n}.0, g_1^{-(n+1)}.0\big)_\st \to \big(0,g_1^{-1}.0\big)_\st,\quad x\mapsto g_1^n.x,
\]
as well as
\[
 \big(g_{q-1}^{-(n+1)}.\infty, g_{q-1}^{-n}.\infty\big)_\st \to \big(g_{q-1}^{-1}.\infty,\infty\big)_\st,\quad x\mapsto g_{q-1}^n.x,
\]
for $n\in\N$. We identify any function $f\colon D_\st\to\C$ with the vector
\[
 \left( f\cdot 1_{(0,g_{q-1}^{-1}.\infty)}, f\cdot 1_{(g_{q-1}^{-1}.\infty, g_1^{-1}.0)}, f\cdot 1_{(g_1^{-1}.0,\infty)}\right).
\]
The associated family of transfer operators is then represented (at first only formally) by the matrices
\begin{equation}\label{transop}
 \mc L_{G,s} = 
\begin{pmatrix}
0 & \sum\limits_{k=2}^{q-2} \tau_s(g_k) & \sum\limits_{n\in\N} \tau_s(g_1^n)
\\
\sum\limits_{n\in\N}\tau_s(g_{q-1}^n) & \sum\limits_{k=2}^{q-2}\tau_s(g_k) & \sum\limits_{n\in\N}\tau_s(g_1^n)
\\
\sum\limits_{n\in\N}\tau_s(g_{q-1}^n) & \sum\limits_{k=2}^{q-2}\tau_s(g_k) & 0
\end{pmatrix}.
\end{equation}
To state a good domain of definition for $\mc L_{G,s}$, we fix appropriate open neighborhoods $\mc D_1, \mc D_r, \mc D_{q-1}$ in $P^1(\C)$ of the closed intervals $[g_1^{-1}.0,\infty], [g_{q-1}^{-1}.\infty, g_1^{-1}.0]$ and $[0,g_{q-1}^{-1}.\infty]$, resp.\@, as in \cite{Moeller_Pohl}, and define
\[
 B(\mc D_j) \sceq \big\{ \text{$f\colon \overline{\mc D}_j \to \C$ continuous} \ \big\vert\ \text{$f\vert_{\mc D_j}$ holomorphic}\big\}
\]
for $j\in \{1,r,q-1\}$. The neighborhoods can and shall be chosen such that $Q.\mc D_{q-1} = \mc D_1$ and $Q.\mc D_r = \mc D_r$. Endowed with the supremum norm, $B(\mc D_j)$ is a Banach space. We let
\[
 B(\mc D) \sceq B(\mc D_1) \times B(\mc D_r) \times B(\mc D_{q-1})
\]
denote the direct product of these Banach spaces. We let $\mc L_{G,s}$ act on $B(\mc D)$ via the matrix representation in \eqref{transop}, whenever the infinite sums in its definition converge. We recall that the Selberg zeta function as well as its meromorphic continuation is denoted by $Z$. The following theorem from \cite{Moeller_Pohl} shows the relation between the zeros of $Z$ and the $1$-eigenfunctions of $\mc L_{G,s}$. 

\begin{thm}[\cite{Moeller_Pohl}]\label{fastMP}
Let $s\in\C$. 
\begin{enumerate}[{\rm (i)}]
\item For $\Rea s > \tfrac12$, the transfer operator $\mc L_{G,s}$ acts on $B(\mc D)$ and defines a nuclear operator of order $0$. 
\item\label{fastMPii} For $\Rea s > 1$ we have $Z(s) = \det(1-\mc L_{G,s})$.
\item The map $s\mapsto \mc L_{G,s}$ extends to a meromorphic map on all of $\C$ with values in nuclear operators of order $0$. Possible poles (all of which are simple) are located at $s=(1-k)/2$, $k\in\N_0$.
\item By meromorphic continuation, we have $Z(s) = \det(1-\mc L_{G,s})$ on all of $\C$. Possible poles are located at $s=(1-k)/2$, $k\in\N_0$. These poles have order at most $4$.
\end{enumerate}
\end{thm}

The operator 
\[
 T_s(Q) \sceq 
\begin{pmatrix}
& & \tau_s(Q)
\\
& \tau_s(Q)
\\
\tau_s(Q) 
\end{pmatrix}
\]
commutes with $\mc L_{G,s}$ and allows, by change of basis, a block-diagonalization
\[
\mc L_{G,s} =
\begin{pmatrix}
\mc L^+_{G,s} 
\\
 & \mc R\mc L^-_{G,s}\mc R
\end{pmatrix}
\]
where, with $m = \lfloor \tfrac{q+1}{2} \rfloor$,
\[
 \mc L_{G,s}^\pm \sceq 
\begin{pmatrix}
\pm \sum\limits_{n\in\N} \tau_s(Qg_{q-1}^n) & \sum\limits_{k=m}^{q-2} \tau_s(g_k) \pm \tau_s(Qg_k)
\\
\sum\limits_{n\in\N} \tau_s(g_{q-1}^n) \pm \tau_s(Qg_{q-1}^n) & \sum\limits_{k=m}^{q-2} \tau_s(g_k)\pm \tau_s(Qg_k)
\end{pmatrix}
\]
for odd $q$, and
\[
\mc L_{G,s}^\pm \sceq
\begin{pmatrix}
\pm \sum\limits_{n\in\N}\tau_s(g_1^nQ) & \sum\limits_{k=m+1}^{q-2}\tau_s(g_k)\pm \tau_s(Qg_k) +\frac12 \tau_s(g_m) \pm \frac12\tau_s(Qg_m)
\\
\sum\limits_{n\in\N}\tau_s(g_{q-1}^n) \pm \tau_s(Qg_{q-1}^n) & \sum\limits_{k=m+1}^{q-2} \tau_s(g_k) \pm \tau_s(Qg_k) + \frac12\tau_s(g_m) \pm \frac12\tau_s(Qg_m)
\end{pmatrix}
\]
for even $q$. The operator $\mc R$ is only a self-inverse base-change. We remark that we corrected here the weights of $\tau_s(g_m)$, $\tau_s(Qg_m)$ in $\mc L^\pm_{G,s}$ compared to \cite{Moeller_Pohl}.

\begin{thm}[\cite{Moeller_Pohl}] \label{MPtwist}
Let $s\in\C$. 
\begin{enumerate}[{\rm (i)}]
\item For $\Rea s > \tfrac12$, the operators $\mc L^\pm_{G,s}$ act on $B(\mc D_{q-1})\times B(\mc D_r)$ and define nuclear operators of order $0$.
\item The maps $s\mapsto \mc L^\pm_{G,s}$ extend to meromorphic maps on all of $\C$ with values in nuclear operators of order $0$ and possible poles at $s=(1-k)/2$, $k\in\N_0$. 
\item The operator $\mc L_{G,s}$ has a $T_s(Q)$-invariant (resp.\@ $T_s(Q)$-anti-invariant) eigenfunction with eigenvalue $\lambda$ if and only if $\mc L^+_{G,s}$ (resp.\@ $\mc L^-_{G,s}$) has an eigenfunction with eigenvalue $\lambda$.
\item We have $Z(s) = \det(1-\mc L^+_{G,s}) \det(1-\mc L^-_{G,s})$.
\end{enumerate}
\end{thm}

\begin{remark}
In \cite{Moeller_Pohl}, we considered transfer operators $\mc L_{H,s}$ respectively $\mc L_{H,s}^\pm$ conjugate to $\mc L_{G,s}$ respectively $\mc L_{G,s}^\pm$ to avoid a change of charts when discussing their nuclearity and other properties. This modification corresponded to actually study Hecke triangle groups conjugate by a certain Cayley transform. For the investigations in this article such a conjugation does not simplify exposition, for which reason we decided to work with the transfer operators for the original Hecke triangle groups.
\end{remark}


\section{Strategy of the proof of Theorem~\ref{mainintro}}\label{sec_strategy}

To explain the strategy of the proof of Theorem~\ref{mainintro} we briefly recall the key points in the proof of Theorem~\ref{fastMP}\eqref{fastMPii}.

\subsubsection*{Periodic geodesics and hyperbolic $\Gamma_q$-conjugacy classes}
As it is well-known, periodic geodesics on $\Gamma_q\backslash\h$ (period length with multiplicity) are in bijection with the $\Gamma_q$-conjugacy classes of hyperbolic elements in $\Gamma_q$ such that primitive periodic geodesics correspond to $\Gamma_q$-conjugacy classes of primitive hyperbolic elements. This results in the well-known algebraic formulation of the Selberg zeta function
\[
 Z(s) = \prod_{ [g]\in [\Gamma_q]_p } \prod_{k=0}^\infty \left(1-N(g)^{-(s+k)}\right).
\]
Here, 
\[
 [g]\sceq [g]_{\Gamma_q}\sceq \{ hgh^{-1}\mid h\in\Gamma_q\}
\]
is the $\Gamma_q$-conjugacy class of $g\in\Gamma_q$, and 
\[
 [\Gamma_q]_p\sceq \{ [g] \mid \text{$g\in\Gamma_q$ primitive hyperbolic} \}
\]
denotes the set of all primitive hyperbolic $\Gamma_q$-conjugacy classes. For a hyperbolic $g\in\Gamma_q$, we denote its norm by $N(g)$, that is the square of its eigenvalue with the larger absolute value. We recall that $N(g) = \exp\ell(\wh\gamma_g)$ if $\wh\gamma_g$ is the geodesic associated to $[g]\in [\Gamma_q]_p$. For future reference we denote by
\[
 [\Gamma_q]_h \sceq \{ [g] \mid \text{$g\in\Gamma_q$ hyperbolic} \}
\]
the set of all hyperbolic $\Gamma_q$-conjugacy classes.

\subsubsection*{Normal forms for representatives of hyperbolic $\Gamma_q$-conjugacy classes}
The discrete dynamical system $(D_\st, G)$ gives rise to normal forms for representatives of hyperbolic $\Gamma_q$-conjugacy classes as explained in the following. Let
\[
 \Gen_G \sceq \{g_2,\ldots, g_{q-2}\} \cup \{ g_1^k, g_{q-1}^k \mid k\in\N\}
\]
denote the set of generators of $(D_\st, G)$. Let $w=w_1\ldots w_n$ be a word of length $n\in\N$ over the alphabet $\Gen_G$ (that is, $w_i\in \Gen_G$). We say that $w$  is \textit{reduced} if $w$ does not contain any subword of the form $g_1^{k_1}g_1^{k_2}$ or $g_{q-1}^{k_1}g_{q-1}^{k_2}$ for $k_1,k_2\in\N$. We say that $w$ is \textit{regular} if $ww$ is reduced.

We define a labeling of the cross section $\wh C_G$ and introduce coding sequences for its elements as follows. Let $v\in C'_G$ and consider the determined geodesic $\gamma_v$ on $\h$. Let $t_0$ be the first return time of $\pi(v)$. From \cite{Moeller_Pohl} we know that there is a \textit{unique} element $g\in\Gamma_q$ such that
\[
 \gamma'_v(t_0) \in g.C'_G.
\]
More precisely, $g^{-1}\in \Gen_G$ and 
\begin{itemize}
\item $g = g_k^{-1}$ if and only if $\gamma_v(\infty) \in (g_k^{-1}.0, g_k^{-1}.\infty)$ for some $k\in \{2,\ldots, q-2\}$,
\item $g = g_1^{-n}$  for $n\in\N$ if and only if $\gamma_v(\infty) \in (g_1^{-n}.0, g_1^{-(n+1)}.0)$, and
\item $g = g_{q-1}^{-n}$ for $n\in\N$ if and only if $\gamma_v(\infty) \in (g_{q-1}^{-(n+1)}.\infty, g_{q-1}^{-n}.\infty)$.
\end{itemize}
We assign the label $g^{-1}$ to $v$, and inherit these labels to $\wh C_G$ via the quotient map $\pi$. Now let $\wh v\in \wh C_G$. We assign to $\wh v$ the $G$-coding sequence $(a_0, a_1,\ldots )$, where $a_n$ is the label of $R^n(\wh v)$ for all $n\in\N_0$.
The following facts are proven in \cite{Moeller_Pohl}:
\begin{itemize}
\item Coding sequences are unique. Distinct elements in $\wh C_G$ have distinct coding sequences.
\item If and only if $\wh v\in \wh C_G$ determines a periodic geodesic $\wh\gamma$ on $\Gamma_q\backslash\h$, the associated $G$-coding sequence $(a_n)_{n\in\N_0}$ is periodic. If $v\sceq \big(\pi\vert_{C'_G}\big)^{-1}(\wh v)$ and $\gamma_v$ is the geodesic on $\h$ determined by $v$, and $(\overline{a_0,\ldots, a_{k-1}}) = (a_n)_{n\in\N_0}$ with $k\in\N$ minimal, then $(a_0\cdots a_{k-1})^{-1}$ is the primitive hyperbolic element which fixes $\gamma_v(\pm \infty)$ and has $\gamma_v(\infty)$ as attracting fixed point. Moreover, the word $a_0\ldots a_{k-1}$ is reduced. 
\item An element $g \in \Gamma_q$ is hyperbolic if and only if $[g]_{\Gamma_q}$ contains a representative of the form 
\begin{equation}\label{repr}
 h_{i_1}\cdots h_{i_n}
\end{equation}
with $h_{i_j}\in \Gen_G$ such that $h_{i_a} \notin\{g_1^k\mid k\in\N\}$ and $h_{i_b}\notin\{g_{q-1}^k\mid k\in\N\}$ for some $a,b\in \{1,\ldots, n\}$ and such that the word $h_{i_1}\ldots h_{i_n}$ is regular. In this case, the representative in \eqref{repr} is unique up to cyclic permutation. Moreover, $g$ is primitive if and only if the representative in \eqref{repr} is not of the form $p^{\frac{\ell}{k}}$ with
\[
 p = h_{i_1}\cdots h_{i_k}
\]
for some $k<\ell$. 
\item If $g\in\Gamma_q$ is primitive hyperbolic, then the representatives in \eqref{repr} arise as follows: Let $\gamma$ be a geodesic on $\h$ such that $\gamma(\infty)$ is the attracting fixed point of $g$ and $\gamma(-\infty)$ is the repelling fixed point. This geodesic is unique up to time-shifts. Then the geodesic $\wh\gamma \sceq \pi(\gamma)$ on $\Gamma_q\backslash\h$ is periodic and intersects as such the cross section $\wh C_G$. Suppose that $\wh v_1,\ldots, \wh v_\ell$ are precisely the intersection vectors. The $G$-coding sequences of these vectors are periodic. All of them have the same minimal period length and they only differ by right-shifts. Let $(a_n)_{n\in\N_0} = (\overline{a_0,\ldots, a_{\ell-1}})$ be one of these $G$-coding sequences. Clearly, $\ell\in\N$ is the minimal period length. Then the representatives in \eqref{repr} are precisely
\begin{equation}\label{primrepr}
 a_0\cdots a_{\ell-1},\quad a_1\cdots a_{\ell-1}a_0,\quad \ldots,\quad a_{\ell-1}a_0\cdots a_{\ell-2}.
\end{equation}
If $g\in\Gamma_q$ is hyperbolic but not primitive and $h$ is the primitive hyperbolic element in $\Gamma_q$ such that $h^n= g$ for some $n\in\N$, then one has to take the $n$-powers of the representatives in \eqref{primrepr}.
\end{itemize}

For $[g]_{\Gamma_q} \in [\Gamma_q]_h$ we consider all of its representatives from \eqref{repr} to be a normal form.

\subsubsection*{Traces of transfer operators and the Selberg zeta function}
By definition, the Fredholm determinant of $\mc L_{G,s}$ is
\[
 \det(1-\mc L_{G,s}) = \exp\left( -\sum_{n=1}^\infty \frac1n \Tr\mc L_{G,s}^n\right),
\]
where in our situation (see \cite{Moeller_Pohl})
\[
 \Tr \mc L_{G,s}^n = \sum_{a\in P_n} \Tr \tau_s(a)
\]
with $P_n$ being the set of normal forms for representatives of length $n$ of hyperbolic $\Gamma_q$-conjugacy classes and 
\[
 \Tr \tau_s(a) = \frac{N(a)^{-s}}{1-N(a)^{-1}}.
\]
For $g\in\Gamma_q$ hyperbolic, let $n=n(g)$ be the maximal element in $\N$ such that there exists $h\in\Gamma_q$ primitive hyperbolic with $h^n=g$. For the Selberg zeta function we have
\[
 \log Z(s) = - \sum_{[g]\in [\Gamma_q]_h} \frac{1}{n(g)} \frac{N(g)^{-s}}{1-N(g)^{-1}}.
\]
Now an easy counting establishes the identity
\[
 \log Z(s) = - \sum_{n=1}^\infty \frac1n \Tr \mc L_{G,s}^n.
\]

In this article we use the billiard flow on $\wt\Gamma_q\backslash\h$ to establish analogous relations between the operators $\mc L_{G,s}^\pm$ and the zeta functions $Z^V_\pm$. As a flow, the billiard flow is identical to the geodesic flow on $\wt\Gamma_q\backslash\h$. In the definition of discrete dynamical systems we will allow weights which reflect whether a periodic geodesic on $\wt\Gamma_q\backslash\h$ has the same period length as the corresponding periodic geodesic on $\Gamma_q\backslash\h$ or half of it. In this way we can accommodate Dirichlet and Neumann boundary value conditions.

The (fast) discrete dynamical systems from which $\mc L^\pm_{G,s}$ arise as transfer operator families will be denoted by $(I_\st, G^Q,\pm)$ for some interval $I$ in $\R$ (which depends on whether $q$ is even or odd) and where $\pm$ defines the choice of weights. To simplify exposition we first construct (slow) discrete dynamical systems $(I_\st, F^Q, \pm)$ in a certain analogy to $(D_\st, F)$. These provide a transfer operator interpretation to the functional equations for odd and even period functions from Section~\ref{sec_MPslow}.

We introduce a few notions and definitions.

\subsubsection*{Hyperbolic elements in $\wt\Gamma_q$}
An element $h\in \wt\Gamma_q$ is called \textit{hyperbolic} if $h^2\in \Gamma_q$ is hyperbolic. The norm of $h$ is defined as $N(h) = N(h^2)^{1/2}$. The element $h$ is called ($\wt\Gamma_q$-)\textit{primitive} hyperbolic if $h$ is not a nontrivial integral power of any hyperbolic element in $\wt\Gamma_q$. For $g \in \wt\Gamma_q$ we use 
\[
[g]\sceq [g]_{\wt\Gamma_q} \sceq \{ hgh^{-1} \mid h\in\wt\Gamma_q\} 
\]
to denote the $\wt\Gamma_q$-conjugacy class of $g$. It will always be clear from the context whether $[g]$ refers to a $\wt\Gamma_q$-conjugacy class or a $\Gamma_q$-conjugacy class. Moreover, we define
\[
 [\wt\Gamma_q]_h \sceq \{ [g]_{\wt\Gamma_q} \mid \text{$g\in \wt\Gamma_q$ hyperbolic}\}
\]
and
\[
 [\wt\Gamma_q]_p \sceq \{ [g]_{\wt\Gamma_q} \mid \text{$g\in \wt\Gamma_q$ primitive hyperbolic}\}.
\]

\subsubsection*{Cross sections for the billiard flow on $\wt\Gamma_q\backslash\h$}
Recall that $\pi$ denotes the quotient map $S\h\to \Gamma_q\backslash S\h$. Let $\pi_Q\colon S\h \to \wt\Gamma_q\backslash S\h$ denote the canonical quotient map. As cross sections for the billiard flow on $\wt\Gamma_q\backslash \h$ we use
\[
 \wh C_{F^Q} \sceq \pi_Q\big(\pi^{-1}\big(\wh C_F\big)\big) \quad\text{and}\quad \wh C_{G^Q} \sceq \pi_Q\big( \pi^{-1}\big(\wh C_G\big)\big).
\]
Note that $1\in\R$ is a fixed point of $Q$ and $Q.[1,\infty)_\st = (0,1]_\st$. Therefore, 
\begin{align*}
C'_{F^Q} &\sceq \{ v\in C'_F \mid \gamma_v(\infty)\in (0,1]_\st\}
\\
&\  = \{ v\in S\h \mid \text{$v$ is based on $i\R^+$, $\gamma_v(\infty)\in (0,1]_\st$}\}
\end{align*}
is a set of representatives for $\wh C_{F^Q}$, and
\[
C'_{G^Q} \sceq \{ v\in C'_G \mid \gamma_v(\infty) \in (0,1]_\st \}.
\]
is one for $\wh C_{G^Q}$. 

\subsubsection*{The significant difference between odd $q$ and even $q$}
A straighforward calculation shows that, for any $k\in\N$, we have
\[
 g_k = \frac{1}{\sin\tfrac\pi q}\bmat{\sin\left(\frac{k}{q}\pi\right)}{-\sin\left(\frac{k+1}{q}\pi\right)}{-\sin\left(\frac{k-1}{q}\pi\right)}{\sin\left(\frac{k}{q}\pi\right)}.
\]
If $q$ is odd, then $1\in\R$ is cuspidal as $g_{\frac{q+1}{2}}^{-1}.\infty=1$. In this case, the union 
\[
 C'_{F^Q} \cup Q.C'_{F^Q} = C'_F
\]
is disjoint, and we can use $C'_{F^Q}$ (and $C'_{G^Q}$) to construct discrete dynamical systems and deduce transfer operator families in analogy to Section~\ref{sec_known}. Moreover, periodic billiards are in bijection with hyperbolic $\wt\Gamma_q$-conjugacy classes. 
The proof of Theorem~\ref{mainintro} for odd $q$ is provided in Section~\ref{odddone} below.

In contrast, if $q$ is even, then $1$ is fixed by the hyperbolic element $g_{\frac{q}{2}}$ and we have
\[
 C'_{F^Q} \cap Q.C'_{F^Q} = \{ v\in C'_F \mid \gamma_v(\infty) = 1\}.
\]
The unit tangent vector $\pi_Q\left(\tfrac{\partial}{\partial x}\vert_i\right) \in \wh C_{F^Q}$ corresponds to the periodic billiard trajectory on $\wt\Gamma_q\backslash\h$ which is represented by the geodesic on $\h$ with endpoints $\pm 1$ (the axis of $g_{\frac{q}{2}}$). This is a so-called boundary periodic geodesic, since it belongs to the boundary of a fundamental domain of $\wt\Gamma_q$ (see Figure~\ref{Forddom}). The periodic billiard trajectory on $\wt\Gamma_q\backslash\h$ corresponds to the two distinct primitive hyperbolic $\wt\Gamma_q$-conjugacy classes $[g_{\frac{q}{2}}]_{\wt\Gamma_q}$ and $[Qg_{\frac{q}{2}}]_{\wt\Gamma_q}$. 

The existence of this boundary periodic geodesic results in the fact that any coding which allows a classification of $\wt\Gamma_q$-primitive hyperbolic elements will have a double-labeling of the vectors $\pi_Q(v)$ with $v\in C'_{F^Q}\cap Q.C'_{F^Q}$. Moreover, any induced discrete dynamical system on (parts of) $\R$ will be a relation rather than a function. Therefore, for even $q$, the proof of Theorem~\ref{mainintro} is more involved. We provide it in Section~\ref{evendone}.


\section{Hecke triangle groups $\Gamma_q$ with $q$ odd}\label{odddone}

Since $q$ is odd and hence $1$ is cuspidal, we have 
\[
 C'_{F^Q} \cap Q.C'_{F^Q} = \emptyset \quad\text{as well as}\quad C'_{G^Q} \cap Q.C'_{G^Q} = \emptyset.
\]
In this situation, we deduce discrete dynamical systems on $(0,1]_\st = (0,1)_\st$ as in Section~\ref{sec_known} using $C'_{F^Q}$ instead of $C'_F$ and $C'_{G^Q}$ instead of $C'_G$. One can even read them off from $(D_\st, F)$ and $(D_\st, G)$. We endow the arising dynamical systems with weights as indicated in Sections~\ref{sec:slowodd} and \ref{sec:fastodd} below. If $(I_\st, H, w)$ is a (weighted) discrete dynamical system with the weighted submaps
\[
 I_{g,\st} \to H(I_{g,\st}),\ x\mapsto g.x \qquad \text{weight: $w_g$},
\]
then the associated transfer operator $\mc L_{H,s}^w$ with parameter $s\in\C$ is (formally) given by
\[
 \mc L_{H,s}^w = \sum w_g \cdot 1_{H(I_{g,\st})} \cdot \tau_s(g),
\]
where $1_{H(I_{g,\st})}$ denotes the indicator function of $H(I_{g,\st})$.
Throughout let 
\[
 m = \left\lfloor \frac{q+1}{2}\right\rfloor = \frac{q+1}{2}
\]
and note that for $k\in\Z$, we have
\[
 Qg_k = g_{q-k}Q.
\]

\subsection{Slow discrete dynamical system}\label{sec:slowodd}
The arising slow discrete dynamical systems $( (0,1)_\st, F^Q, \pm)$ are given by the weighted submaps
\begin{align*}
& (g_k^{-1}.0,g_k^{-1}.1)_\st \to (0,1)_\st,\quad x\mapsto g_k.x && \text{weight: $+1$}
\intertext{and}
& (g_k^{-1}.1,g_k^{-1}.\infty)_\st \to (0,1)_\st,\quad x\mapsto Qg_k.x && \text{weight: $\pm 1$}
\end{align*}
for $k=m,\ldots, q-1$. The associated transfer operator with parameter $s\in\C$ is
\[
 \mc L^\pm_{F^Q,s} = \sum_{k=m}^{q-1} \tau_s(g_k) \pm \tau_s(Qg_k),
\]
defined on $\Fct( (0,1)_\st; \C)$. Clearly, $\mc L^\pm_{F^Q,s}$ also makes sense as an operator on $\Fct( (0,1);\C)$ and on $\Fct(\R_{>0};\C)$. Comparing to Section~\ref{sec_known}, we see that this construction provides a transfer operator interpretation of the functional equations for odd and even Maass cusp forms. The space $\Fct( (0,1); \C)$ is, in a certain sense, a fundamental domain for $\mc L^\pm_{F^Q,s}$.

\subsection{Fast discrete dynamical system}\label{sec:fastodd}
The two fast discrete dynamical systems $( (0,1)_\st, G^Q,\pm)$ are given by the submaps
\begin{align*}
& (g_k^{-1}.0, g_k^{-1}.1)_\st \to (0,1)_\st,\quad x\mapsto g_k.x &&\text{weight: $+1$}
\\
& (g_k^{-1}.1, g_k^{-1}.\infty)_\st \to (0,1)_\st,\quad x\mapsto Qg_k.x && \text{weight: $\pm 1$}
\end{align*}
for $k=m,\ldots, q-2$, and
\begin{align*}
(g_{q-1}^{-(n+1)}.\infty, g_{q-1}^{-n}.1)_\st \to (g_{q-1}^{-1}.\infty, 1)_\st,\quad x\mapsto g_{q-1}^n.x &&\text{weight: $+1$}
\\
(g_{q-1}^{-n}.1, g_{q-1}^{-n}.\infty)_\st \to (0,1)_\st,\quad x\mapsto Qg_{q-1}^n.x && \text{weight: $\pm 1$}
\end{align*}
for $n\in\N$.

Then the associated transfer operator with parameter $s\in\C$ is (formally)
\[
 \mc L^\pm_{G^Q,s} = 
\begin{pmatrix}
\pm \sum\limits_{n\in\N} \tau_s(Qg_{q-1}^n) & \sum\limits_{k=m}^{q-2} \tau_s(g_k) \pm \tau_s(Qg_k)
\\
\sum\limits_{n\in\N} \tau_s(g_{q-1}^n) \pm \tau_s(Qg_{q-1}^n) & \sum\limits_{k=m}^{q-2} \tau_s(g_k) \pm \tau_s(Qg_k)
\end{pmatrix}.
\]
As domain of definition we use the Banach space $B(\mc D_{q-1})\times B(\mc D_r)$ from Section~\ref{sec_known} and recall from Theorem~\ref{MPtwist} that $\mc L^\pm_{G^Q,s}$ is a well-defined operator for $\Rea s > \tfrac12$ and that the map $s\mapsto \mc L^\pm_{G^Q,s}$ extends meromorphically to all of $\C$.

\subsection{Zeta functions and Fredholm determinants of transfer operators}

We set
\begin{align*}
Z_+(s) & \sceq \prod_{[g]\in [\wt\Gamma_q]_p}\prod_{k=0}^\infty \left(1-\det g^k\cdot N(g)^{-(s+k)}\right)
\intertext{and}
Z_-(s) & \sceq \prod_{[g]\in [\wt\Gamma_q]_p}\prod_{k=0}^\infty \left(1-\det g^{k+1}\cdot N(g)^{-(s+k)}\right).
\end{align*}
Both functions converge absolutely for $\Rea s > 1$. In this section we prove the following relation.

\begin{thm}\label{mainodd}
For $\Rea s>1$ we have 
\[
 Z_{\pm}(s) = \det(1-\mc L^\pm_{G^Q,s}).
\]
Moreover, the zeta functions $Z_{\pm}$ admit a meromorphic continuation to all of $\C$ with possible poles at $s=(1-k)/2$, $k\in\N_0$.
\end{thm}

Let
\[
 \Gen_{G^Q} \sceq \{g_m,\ldots, g_{q-2}, Qg_m,\ldots, Qg_{q-2}\} \cup \{ g_{q-1}^k, Qg_{q-1}^k\mid k\in\N\}
\]
denote the set of generators of $((0,1)_\st, G^Q, \pm)$. For a word $h=h_1\ldots h_n$ of length $n\in\N$ with $h_j\in \Gen_{G^Q}$ and $s\in\C$, we define
\begin{equation}\label{bs}
 b_s^\pm(h) \sceq (\pm 1)^\eps \tau_s(h),
\end{equation}
where 
\[
\eps = \eps(h) \sceq \#\left\{ j \in \{1,\ldots n\} \left\vert\ h_j \in \{Qg_m,\ldots Qg_{q-2}, Qg_{q-1}^\ell\mid \ell\in\N\} \right.\right\}.
\]
On the right hand side of \eqref{bs}, we identified the word $h=h_1\ldots h_n$ with the element $h=h_1\cdots h_n$. We say that $h$ is \textit{reduced} if and only if $h$ does not contain any subword of the form $g_{q-1}^{m_1}g_{q-1}^{m_2}$ or $Qg_{q-1}^{m_1}g_{q-1}^{m_2}$ for $m_1,m_2\in\N$. We let $W_n^\redu(\Gen_{G^Q})$ denote the set of all reduced words over the alphabet $\Gen_{G^Q}$ of length $n$, and set
\[
 W_*^\redu(\Gen_{G^Q})=\bigcup_{n\in\N} W_n^\redu(\Gen_{G^Q}).
\]
We say that $h$ is \textit{regular} if $hh$ is reduced.

For $n\in\N$, let 
\begin{itemize}
\item $B_1^n$ denote the words in $W_n^\redu\big(\Gen_{G^Q}\big)$ which end with $g_{q-1}^\ell$ or $Qg_{q-1}^\ell$ for some $\ell\in\N$ and do not begin with $g_{q-1}^k$ for any $k\in\N$,
\item $B_2^n$ denote the words in $W_n^\redu(\Gen_{G^Q}\big)$ which end with an element of $\{g_k, Qg_k\mid k=m,\ldots, q-2\}$ and do not begin with $g_{q-1}^k$ for any $k\in\N$,
\item $B_3^n$ denote the words in $W_n^\redu(\Gen_{G^Q}\big)$ which end with $g_{q-1}^\ell$ or $Qg_{q-1}^\ell$ for some $\ell\in\N$, and
\item $B_4^n$ denote the words in $W_n^\redu(\Gen_{G^Q}\big)$ which end with an element of $\{g_k, Qg_k\mid k=m,\ldots, q-2\}$.
\end{itemize}

Note that the elements of $B_1^n$ and $B_4^n$ are regular. A straighforward induction immediately proves the following formula for the powers of $\mc L^{\pm}_{G^Q,s}$.

\begin{lemma}
For $n\in\N$ we have
\[
\left(\mc L^\pm_{G^Q,s}\right)^n = 
\begin{pmatrix}
\sum\limits_{a\in B_1^n} b_s^\pm(a) & \sum\limits_{a\in B_2^n} b_s^\pm(a)
\\
\sum\limits_{a\in B_3^n} b_s^\pm(a) & \sum\limits_{a\in B_4^n} b_s^\pm(a) 
\end{pmatrix}.
\]
\end{lemma}

The following proposition is a consequence of the normal forms for representatives of hyperbolic $\wt\Gamma_q$-conjugacy classes induced by $( (0,1)_\st, G^Q, \pm)$.

\begin{prop}\label{hyp}
\begin{enumerate}[{\rm (i)}]
\item\label{hypi} Let $[w]\in [\wt\Gamma_q]_h$. Then there exists a regular word in the set $W_*^\redu(\Gen_{G^Q})$ which represents $[w]$. The length of this word is unique, say $\ell(w)$. Let $[h]\in [\wt\Gamma_q]_p$ and $n\in\N$ be the unique elements such that $[h^n]=[w]$. Then there are exactly $\ell(h)$ regular representatives of $[w]$ in $W_*^\redu(\Gen_{G^Q})$.
\item\label{hypii} For any $n\in\N$, the union $B_1^n\cup B_4^n$ is the set of the regular representatives in $W_*^\redu(\Gen_{G^Q})$ of the elements in $[\wt\Gamma_q]_h$ of length $n$.
\end{enumerate}
\end{prop}

\begin{proof}
We assign labels and $G^Q$-coding sequences to the elements of $\wh C_{G^Q}$ in an analogous way as described in Section~\ref{sec_strategy} for $\wh C_G$. Since 
\[
 C'_{G^Q} \cup Q.C'_{G^Q} = C'_G
\]
is a disjoint union, each vector in $\wh C_{G^Q}$ has a unique label. The relation between the maps $G$ and $G^Q$ implies that we have an easy algorithm to translate $G$-coding sequences into $G^Q$-coding sequences and vice versa, which we explain in the following. Let $(a_n)_{n\in\N_0}$ be a $G$-coding sequence of, say, the vector $\wh v\in \wh C_G$ and denote the $G^Q$-coding sequence of $\wt v\sceq \pi_Q(\pi^{-1}(\wh v)) \in \wh C_{G^Q}$ by $(b_n)_{n\in\N_0}$. Now we proceed iteratively to determine $(b_n)_{n\in\N_0}$ from $(a_n)_{n\in\N_0}$ as follows:
\begin{itemize}
\item If $a_0 \in \{g_m,\ldots, g_{q-2}\}\cup \{g_{q-1}^k\mid k\in\N\}$, then set $b_0\sceq a_0$ and $a_1'\sceq a_1$.
\item If $a_0 \in \{g_2,\ldots, g_{m-1}\}\cup \{g_1^k\mid k\in\N\}$, then set $b_0\sceq a_0Q$ and $a_1'\sceq Qa_1$. 
\end{itemize}
Now iteratively for $j=1,2,\ldots$:
\begin{itemize}
\item If $a_j'\in \{g_m, \ldots, g_{q-2}, Qg_m,\ldots, Qg_{q-2}\}\cup \{g_{q-1}^k, Qg_{q-1}^k \mid k\in\N\}$, then set $b_j\sceq a_j'$ and $a_{j+1}'\sceq a_{j+1}$. 
\item If $a_j'\in \{g_2,\ldots, g_{m-1}, Qg_2,\ldots, Qg_{m-1}\} \cup \{g_1^k, Qg_1^k\mid k\in\N\}$, then set $b_j\sceq a_j'Q$ and $a_{j+1}'\sceq Qa_{q+1}$. 
\end{itemize}
For the translation of $G^Q$-coding sequences into $G$-coding sequences we use this algorithm in the opposite direction. This proves that $G$-coding sequences and $G^Q$-coding sequences are in bijection. Since $\wt v\in \wh C_{G^Q}$ determines a periodic billiard if and only if $\wh v$ determines a periodic geodesic on $\Gamma_q\backslash\h$, periodic $G^Q$-coding sequences correspond to 
periodic billiards in the same way as periodic $G$-coding sequences correspond to periodic geodesics on $\Gamma_q\backslash\h$. In particular, it follows that $w\in\wt\Gamma_q$ is hyperbolic if and only if $[w]_{\wt\Gamma_q}$ contains a representative of the form
\begin{equation}\label{reprQ}
 w_{i_1}\cdots w_{i_\ell}
\end{equation}
with $w_{i_j}\in \Gen_{G^Q}$ such that $w_{i_a}\notin \{ g_{q-1}^k\mid k\in\N\}$ for some $a\in \{1,\ldots, \ell\}$ and the word $w_{i_1}\ldots w_{i_\ell}$ is regular. In this case, the representative in \eqref{reprQ} is unique up to cyclic permutation. The element $w$ is $\wt\Gamma_q$-primitive if and only if the representative in \eqref{reprQ} is not of the form $h^{\frac{\ell}{k}}$ with
\[
 h=w_{i_1}\cdots w_{i_k}
\]
for some $k<\ell$. This proves \eqref{hypi}. The statement \eqref{hypii} follows now from the fact that $B_1^n\cup B_4^n$ contains all regular words in $W_n^\redu(\Gen_{G^Q})$.
\end{proof}

\begin{lemma}\label{trace}
For $n\in\N$ and $a\in B_1^n\cup B_4^n$ we have
\begin{align*}
\Tr b_s^-(a) & = \frac{\det a\cdot N(a)^{-s}}{1-\det a\cdot N(a)^{-1}}
\intertext{and}
\Tr b_s^+(a) & = \frac{N(a)^{-s}}{1-\det a\cdot N(a)^{-1}}.
\end{align*}
The value of $\Tr b_s^\pm(a)$ is an invariant of the conjugacy class $[a]_{\wt\Gamma_q}$.
\end{lemma}

\begin{proof}
This is a standard calculation, which follows from a general formula for traces of composition operators (see e.g.\@ \cite{Ruelle_zeta}).
\end{proof}

\begin{proof}[Proof of Theorem~\ref{mainodd}]
For $w\in\wt\Gamma_q$ hyperbolic, let $n(w)\in\N$ denote the unique number such that $w=h^{n(w)}$ for some primitive hyperbolic element $h\in \wt\Gamma_q$. Further recall that $\ell(w)$ denotes the length of $w$ in the sense of Proposition~\ref{hyp}. We have
\begin{align*}
\log Z_-(s) & = -\sum_{\ell=1}^\infty \frac{1}{\ell} \sum_{\stackrel{[w]\in [\wt\Gamma_q]_h}{\ell(w)=\ell}} \frac{\ell(w)}{n(w)} \cdot \frac{\det w\cdot N(w)^{-s}}{1-\det w\cdot N(w)^{-1}}.
\end{align*}
Let $[w]\in [\wt\Gamma_q]_h$ with $\ell\sceq \ell(w)$. From Proposition~\ref{hyp} and Lemma~\ref{trace} it follows that 
\[
 \sum_{\stackrel{a\in B_1^\ell\cup B_4^\ell}{[a]=[w]}} \Tr b_s^-(a) = \frac{\ell(w)}{n(w)} \cdot\frac{\det w\cdot N(w)^{-s}}{1-\det w\cdot N(w)^{-1}}.
\]
Therefore,
\begin{align*}
\log Z_-(s) & = -\sum_{\ell=1}^\infty \frac{1}{\ell} \sum_{a\in B_1^\ell\cup B_4^\ell} \Tr b_s^-(a) = -\sum_{\ell=1}^\infty \frac{1}{\ell} \Tr\left( \mc L_{G^Q,s}^-\right)^\ell.
\end{align*}
This completes the identity for $Z_-$. The statements on the meromorphic continuation follow immediately from Theorem~\ref{MPtwist}. The proof for $Z_+$ is analogous.
\end{proof}

\subsection{Zeta functions and eigenvalues}\label{Venkov_odd}

Venkov's Selberg-type zeta function for the Dirichlet boundary value problem is
\[
Z_-^V(s)  = \prod_{[g]\in [\Gamma_q]_p}\prod_{k=0}^\infty \left(1-N(g)^{-(s+k)}\right)^2 \cdot \prod_{\stackrel{[h]\in [\wt\Gamma_q]_p}{\det h=-1}} \prod_{\ell=0}^\infty \left(\frac{1+N(h)^{-(s+\ell)}}{1-N(h)^{-(s+\ell)}}\right)^{2(-1)^\ell},
\]
and the Selberg-type zeta function for the Neumann boundary value problem is
\[
Z_+^V(s) = \prod_{[g]\in [\Gamma_q]_p}\prod_{k=0}^\infty \left(1-N(g)^{-(s+k)}\right)^2 \cdot \prod_{\stackrel{[h]\in [\wt\Gamma_q]_p}{\det h=-1}} \prod_{\ell=0}^\infty \left(\frac{1-N(h)^{-(s+\ell)}}{1+N(h)^{-(s+\ell)}}\right)^{2(-1)^\ell}.
\]
The zeta functions $Z^V_\pm$ and $Z_\pm$ are closely related, as the following lemma shows. Here, we use $Z_\pm$ to also denote their meromorphic continuations.

\begin{lemma}\label{zetasident}
For all $s\in\C$, we have $Z^V_{\pm}(s)= Z_{\pm}(s)^4$. 
\end{lemma}

\begin{proof}
It suffices to prove the claimed identity for $\Rea s>1$. We only provide a proof of $Z^V_-(s) = Z_-(s)^4$. The argumentation for $Z^V_+$ and $Z_+$ is analogous. The proof is based on the relation between conjugacy classes of $\Gamma_q$-primitive hyperbolic elements and $\wt\Gamma_q$-primitive hyperbolic elements. Invoking their characterization in the proof of Proposition~\ref{hyp}, one easily sees that if $[g]_{\Gamma_q} \in [\Gamma_q]_p$ is not $\wt\Gamma_q$-primitive, then there is a primitive hyperbolic $h\in\wt\Gamma_q$ with $\det h=-1$ and $h^2=g$. If $[g]_{\Gamma_q}\in [\Gamma_q]_p$ and $g$ is $\wt\Gamma_q$-primitive, then also $QgQ$ is primitive hyperbolic and $[QgQ]_{\Gamma_q}\not=[g]_{\Gamma_q}$ but $[QgQ]_{\wt\Gamma_q} = [g]_{\wt\Gamma_q}$. Conversely, if $[h]_{\wt\Gamma_q} \in [\wt\Gamma_q]_p$ and $\det h=-1$, then $[h^2]_{\Gamma_q}\in [\Gamma_q]_p$. If $[h]_{\wt\Gamma_q} \in [\wt\Gamma_q]_p$ with $\det h = 1$, then $[h]_{\wt\Gamma_q}\cap \Gamma_q$ splits 
into the two distinct (primitive) classes $[h]_{\Gamma_q}, [QhQ]_{\Gamma_q}\in [\Gamma_q]_p$. Therefore we have
\begin{align*}
Z_-^V(s) &  = \prod_{\stackrel{[a]\in [\wt\Gamma_q]_p}{\det a = 1}} \prod_{k=0}^\infty \left(1-N(a)^{-(s+k)}\right)^4 \cdot \prod_{\stackrel{[b]\in [\wt\Gamma_q]_p}{\det b=-1}}\prod_{n=0}^\infty \left(1-N(b^2)^{-(s+n)}\right)^2 
\\
& \quad \times \prod_{\stackrel{[h]\in [\wt\Gamma_q]_p}{\det h=-1}} \prod_{\ell=0}^\infty \left(\frac{1+N(h)^{-(s+\ell)}}{1-N(h)^{-(s+\ell)}}\right)^{2(-1)^\ell}
\\
& = \prod_{\stackrel{[a]\in [\wt\Gamma_q]_p}{\det a = 1}} \prod_{k=0}^\infty \left(1-N(a)^{-(s+k)}\right)^4 \cdot  \prod_{\stackrel{[h]\in[\wt\Gamma_q]_p}{\det h=-1}}\prod_{\ell=0}^\infty \left(1-\det h^{\ell+1}\cdot N(h)^{-(s+\ell)}\right)^4
\\
& = Z_-(s)^4.
\end{align*}
This completes the proof.
\end{proof}

\begin{proof}[Proof of Theorem~\ref{mainintro} for odd $q$]
The claims are now immediately implied by Theorem~\ref{Venkov}, Lemma~\ref{zetasident} and Theorem~\ref{mainodd}.
\end{proof}

\section{Hecke triangle groups $\Gamma_q$ with $q$ even}\label{evendone}

For even $q$, the point $1\in\R$ is fixed by the hyperbolic element $g_{\frac{q}{2}}$ and
\[
 C'_{F^Q} \cap Q.C'_{F^Q} = C'_{G^Q} \cap Q.C'_{G^Q} = \{v\in C'_F \mid \gamma_v(\infty) = 1 \}.
\]
This nontrivial intersection is caused by the existence of boundary periodic geodesics, or, in more algebraic terms, the fact that the primitive hyperbolic $\wt\Gamma_q$-conjugacy classes $[g_{\frac{q}{2}}]$ and $[Qg_{\frac{q}{2}}]$ correspond to the same primitive periodic billiard on $\wt\Gamma_q\backslash\h$. To handle these non-uniquenesses in the relation between periodic billiards and hyperbolic $\wt\Gamma_q$-conjugacy classes correctly in a transfer operator approach, we have to allow double-labelings and non-unique coding sequences of the elements of the cross sections $\wh C_{F^Q}$ and $\wh C_{G^Q}$. For the arising (weighted) discrete dynamical systems this means that they are relations rather than functions. For notational purposes we use the triples
\[
 \text{(initial point, final point, weight)}.
\]
The weights will be of the form $c(\pm 1)^\eta$ for some $c>0$ and $\eta\in \{0,1\}$. The choice $\eta=1$ corresponds to Dirichlet boundary value conditions, and the choice $\eta=0$ corresponds to Neumann boundary value conditions. The value of $c$ takes care of the amount of multiple-coding. The underlying idea to handle the necessary double-coding in a natural way is to start by developing separately discrete dynamical systems (given by relations) for any fixed reasonable choice of sets of representatives and domains of definition and then averaging over all these possibilities. In this way we achieve a certain symmetrization of the construction.

The arising relations will be of the form
\[
 H = \{ (x,h.x, w_{I_h}) \mid x\in I_{h,\st},\ h\in J_H\},
\]
where $J_H$ is some subset of $\wt\Gamma_q$ and the $I_h$ are some intervals in $\R$. The associated transfer operator $\mc L_{H,s}$ with parameter $s\in\C$ is (formally) defined as
\[
 \mc L_{H,s}f(x) = \sum_{(h^{-1}.x, x, w)\in H} w \cdot \tau_s(h)f(x).
\]
Throughout let
\[
 m=\left\lfloor \frac{q+1}{2} \right\rfloor = \frac{q}2.
\]

\subsection{Slow discrete dynamical system}\label{sloweven}
We let
\[
 Z_1 \sceq (0,1]_\st \quad\text{and}\quad Z_2 \sceq (0,g_m^{-1}.0)_\st \cup [1,g_m^{-1}.\infty)_\st
\]
and set
\begin{align*}
C'_{F^Q,1} &\sceq C'_{F^Q} = \{ v\in C'_F \mid \gamma_v(\infty) \in Z_1\}
\intertext{and}
C'_{F^Q,2} &\sceq \{ v\in C'_F \mid \gamma_v(\infty) \in Z_2 \}.
\end{align*}
For $a,b\in\{1,2\}$ and $\eta\in \{0,1\}$ let $F^{Q,\eta}_{a,b}$ denote the relation which arises by using $C'_{F^Q,a}$ to induce a dynamical system on $Z_b$ with Dirichlet ($\eta = 1$) respectively Neumann ($\eta = 0$) boundary values. The symmetrized relation $F^{Q,\eta}$ (which provides the transfer operator) is then given by taking the union 
\[
F^{Q,\eta}_{1,1} \cup F^{Q,\eta}_{1,2} \cup F^{Q,\eta}_{2,1} \cup F^{Q,\eta}_{2,2}
\]
as relations, and averaging over the weights. By straightforward modifications of $(D_\st,F)$, we find
\begin{align*}
F^{Q,\eta}_{1,1} & = \big\{ (x, g_k.x, 1) \ \big\vert\ x\in\big(g_k^{-1}.0, g_k^{-1}g_m^{-1}.0\big)_\st,\ k=m,\ldots, q-1 \big\}
\\
& \quad\cup \big\{ (x, g_k.x, 2) \ \big\vert\ x\in\big(g_k^{-1}g_m^{-1}.0, g_k^{-1}.1\big)_\st,\ k=m,\ldots, q-1\big\}
\\
& \quad\cup \big\{ (x, Qg_k.x, (\pm 1)^\eta 2) \ \big\vert\ x\in \big(g_k^{-1}.1, g_k^{-1}g_m^{-1}.\infty\big)_\st,\ k=m+1,\ldots, q-1\big\}
\\
& \quad\cup \big\{ (x, Qg_k.x, (\pm 1)^\eta) \ \big\vert\ x\in \big(g_k^{-1}g_m^{-1}.\infty, g_k^{-1}.\infty\big)_\st,\ k=m+1,\ldots, q-1\big\}
\\
& \quad\cup \big\{ \big(g_k^{-1}.1,1,1\big), \big(g_k^{-1}.1, 1, (\pm 1)^\eta\big) \ \big\vert\ k=m+1,\ldots, q-1\big\}
\\
& \quad\cup \big\{ \big(g_m^{-1}.1,1,\tfrac12\big), \big(g_m^{-1}.1, 1, (\pm 1)^\eta\tfrac12\big) \big\},
\end{align*}
\begin{align*}
F^{Q,\eta}_{2,2} & = \big\{ (x,g_k.x,1) \ \big\vert\ x\in\big(g_k^{-1}.0, g_k^{-1}g_m^{-1}.0\big)_\st,\ k=m+1,\ldots, q-1\big\}
\\
& \quad\cup \big\{ (x, Qg_k.x, (\pm 1)^\eta 2) \ \big\vert\ x\in \big(g_k^{-1}g_m^{-1}.0, g_k^{-1}.1\big)_\st,\ k=m+1,\ldots, q-1\big\}
\\
& \quad\cup \big\{ (x,g_k.x, 2) \ \big\vert\ x\in\big(g_k^{-1}.1, g_k^{-1}g_m^{-1}.\infty\big)_\st,\ k=m+1,\ldots, q-1\big\}
\\
& \quad\cup \big\{ (x, Qg_k.x, (\pm 1)^\eta) \ \big\vert\ x\in \big(g_k^{-1}g_m^{-1}.\infty, g_k^{-1}.\infty\big)_\st,\ k=m+1,\ldots, q-1\big\}
\\
& \quad\cup \big\{ \big(g_k^{-1}.1,1,1\big), \big(g_k^{-1}.1,1,(\pm 1)^\eta\big) \ \big\vert\ k=m+1,\ldots, q-1\big\}
\\
& \quad\cup \big\{ (x, g_m.x, 2) \ \big\vert\ x\in\big(1,g_m^{-2}.\infty\big)_\st\big\}
\\
& \quad\cup \big\{ (x, Qg_m.x, (\pm 1)^\eta) \ \big\vert\ x\in \big(g_m^{-2}.\infty, g_m^{-1}.\infty\big)_\st\big\}
\\
& \quad\cup \big\{ \big(g_m^{-1}.1,1,\tfrac12\big), \big(g_m^{-1}.1,1,(\pm 1)^\eta\tfrac12\big) \big\},
\end{align*}
\begin{align*}
F^{Q,\eta}_{1,2} & = \big\{ (x,g_k.x,1) \ \big\vert\ x\in \big(g_k^{-1}.0, g_k^{-1}g_m^{-1}.0\big)_\st,\ k=m+1,\ldots, q-1\big\}
\\
& \quad \cup \big\{ (x, Qg_k.x, (\pm 1)^\eta 2) \ \big\vert\ x\in\big(g_k^{-1}g_m^{-1}.0,g_k^{-1}.1\big)_\st,\ k=m+1,\ldots, q-1 \big\}
\\
& \quad \cup \big\{ (x, g_k.x, 2) \ \big\vert\ x\in \big(g_k^{-1}.1, g_k^{-1}g_m^{-1}.\infty\big)_\st,\ k=m+1,\ldots, q-1\big\}
\\
& \quad \cup \big\{ (x, Qg_k.x, (\pm 1)^\eta) \ \big\vert\ x\in\big(g_k^{-1}g_m^{-1}.\infty, g_k^{-1}.\infty\big)_\st, k=m+1,\ldots, q-1\big\}
\\
& \quad \cup \big\{ \big(g_k^{-1}.1, 1, 1\big), \big(g_k^{-1}.1,1, (\pm 1)^\eta \big) \ \big\vert\ k=m+1,\ldots, q-1\big\}
\\
& \quad \cup \big\{ (x, g_m.x, 1) \ \big\vert\ x\in\big(g_m^{-1}.0,g_m^{-2}.0\big)_\st \big\}
\\
& \quad \cup \big\{ (x, Qg_m.x, (\pm 1)^\eta 2) \ \big\vert\ x\in\big(g_m^{-2}.0,1\big)_\st \big\}
\\
& \quad \cup \big\{ \big(g_m^{-1}.1, 1, \tfrac12\big), \big(g_m^{-1}.1,1, (\pm 1)^\eta \tfrac12\big) \big\},
\end{align*}
and
\begin{align*}
F^{Q,\eta}_{2,1} & = \big\{ (x,g_k.x, 1) \ \big\vert\ x\in\big(g_k^{-1}.0, g_k^{-1}g_m^{-1}.0\big)_\st,\ k=m+1,\ldots, q-1\big\}
\\
& \quad\cup \big\{ (x, g_k.x, 2) \ \big\vert\ x\in\big(g_k^{-1}g_m^{-1}.0, g_k^{-1}.1\big)_\st,\ k=m+1,\ldots, q-1\big\}
\\
& \quad\cup \big\{ (x, Qg_k.x, (\pm 1)^\eta 2) \ \big\vert\ x\in\big(g_k^{-1}.1, g_k^{-1}g_m^{-1}.\infty\big)_\st,\ k=m+1,\ldots, q-1\big\}
\\
& \quad\cup \big\{ (x, Qg_k.x, (\pm 1)^\eta) \ \big\vert\ x\in\big(g_k^{-1}g_m^{-1}.\infty, g_k^{-1}.\infty\big)_\st,\ k=m+1,\ldots, q-1\big\}
\\
& \quad\cup \big\{ \big(g_k^{-1}.1, 1, 1\big), \big(g_k^{-1}.1,1, (\pm 1)^\eta \big) \ \big\vert\ k=m+1,\ldots, q-1\big\}
\\
& \quad\cup \big\{ (x, Qg_m.x, (\pm 1)^\eta 2) \ \big\vert\ x\in \big(1,g_m^{-2}.\infty\big)_\st \big\}
\\
& \quad\cup \big\{ (x, Qg_m.x, (\pm 1)^\eta) \ \big\vert\ x\in \big(g_m^{-2}.\infty, g_m^{-1}.\infty\big)_\st \big\}
\\
& \quad\cup \big\{ \big(g_m^{-1}.1, 1, \tfrac12\big), \big(g_m^{-1}.1,1, (\pm 1)^\eta \tfrac12\big)\big\}.
\end{align*}
Therefore,
\begin{align*}
F^{Q,\eta} & = \big\{ (x, g_k.x, 1) \ \big\vert\ x\in\big(g_k^{-1}.0, g_k^{-1}g_m^{-1}.0\big)_\st,\ k=m+1,\ldots, q-1\big\}
\\
& \quad\cup \big\{ (x,g_k.x, 1) \ \big\vert\ x\in\big(g_k^{-1}g_m^{-1}.0, g_k^{-1}.1\big)_\st,\ k=m+1,\ldots, q-1\big\}
\\
& \quad\cup \big\{ (x, Qg_k.x, (\pm 1)^\eta) \ \big\vert\ x\in\big(g_k^{-1}g_m^{-1}.0, g_k^{-1}.1\big)_\st,\ k=m+1,\ldots, q-1\big\}
\\
& \quad\cup \big\{ (x, Qg_k.x, (\pm 1)^\eta) \ \big\vert\ x\in\big(g_k^{-1}.1, g_k^{-1}g_m^{-1}.\infty\big)_\st,\ k=m+1,\ldots, q-1\big\}
\\
& \quad\cup \big\{ (x, g_k.x, 1) \ \big\vert\ x\in \big(g_k^{-1}.1, g_k^{-1}g_m^{-1}.\infty\big)_\st,\ k=m+1,\ldots, q-1 \big\}
\\
& \quad\cup \big\{ (x, Qg_k.x, (\pm 1)^\eta) \ \big\vert\ x\in \big(g_k^{-1}g_m^{-1}.\infty, g_k^{-1}.\infty\big)_\st,\ k=m+1,\ldots, q-1\big\}
\\
& \quad\cup \big\{ \big(g_k^{-1}.1,1,1\big), \big(g_k^{-1}.1,1,(\pm 1)^\eta \big) \ \big\vert\ k=m+1,\ldots, q-1\big\}
\\
& \quad\cup \big\{ \big(x, g_m.x, \tfrac12\big) \ \big\vert\ x\in\big(g_m^{-1}.0, g_m^{-2}.0\big)_\st \big\}
\\
& \quad\cup \big\{ \big(x, g_m.x, \tfrac12\big) \ \big\vert\ x\in\big(g_m^{-2}.0, 1\big)_\st \big\}
\\
& \quad\cup \big\{ \big(x, Qg_m.x, (\pm 1)^\eta\tfrac12\big) \ \big\vert\ x\in \big(g_m^{-2}.0, 1\big)_\st \big\}
\\
& \quad\cup \big\{ \big(x, g_m.x, \tfrac12\big) \ \big\vert\ x\in\big(1, g_m^{-2}.\infty\big)_\st\big\}
\\
& \quad\cup \big\{ \big(x, Qg_m.x, (\pm 1)^\eta\tfrac12\big) \ \big\vert\ x\in \big(1, g_m^{-2}.\infty\big)_\st\big\}
\\
& \quad\cup \big\{ \big(x, Qg_m.x, (\pm 1)^\eta\tfrac12\big) \ \big\vert\ x\in \big(g_m^{-2}.\infty, g_m^{-1}.\infty\big)_\st \big\}
\\
& \quad\cup \big\{ \big(g_m^{-1}.1,1,\tfrac12\big), \big(g_m^{-1}.1,1,(\pm 1)^\eta \tfrac12\big) \big\}.
\end{align*}

We set $\mc L^+_{F^Q,s} \sceq \mc L_{F^{Q,0},s}$ and $\mc L^-_{F^Q,s} = \mc L_{F^{Q,1},s}$. Then the associated transfer operator with parameter $s\in\C$ becomes
\begin{align*}
\mc L^\pm_{F^Q,s} & = \sum_{k=m+1}^{q-1} \tau_s(g_k) \pm \tau_s(Qg_k)  + \frac12\tau_s(g_m) \pm \frac12\tau_s(Qg_m),
\end{align*}
defined on $\Fct( (0,g_m^{-1}.\infty)_\st;\C)$. We note that the elements of the form $(g_k^{-1}.1, 1, w)$ with its weights perfectly fill the gaps between the definitions on $(0,g_k^{-1}.1)_\st$ and $(g_k^{-1}.1,g_m^{-1}.\infty)_\st$. This formula for $\mc L^\pm_{F^Q,s}$ clearly makes sense as an operator on $\Fct( (0,g_m^{-1}.\infty);\C)$ and on $\Fct( \R_{>0};\C)$. As for odd $q$, we see that this construction provides a transfer operator interpretation of the functional equations for odd respectively even Maass cusp forms from Section~\ref{sec_known}.

\subsection{Fast discrete dynamical system}

When we modify the discrete dynamical system $(D_\st, G)$ analogously to the construction in Section~\ref{sloweven} to deduce relations $G^{Q,\eta}$ using
\begin{align*}
C'_{G^Q,1} & \sceq \{ v\in C'_G \mid \gamma_v(\infty) \in (0,1]_\st\}
\intertext{and}
C'_{G^Q,2} & \sceq \{ v\in C'_G \mid \gamma_v(\infty) \in (0,g_m^{-1}.0)_\st\cup [1,g_m^{-1}.\infty)_\st \}
\end{align*}
in place of $C'_{F^Q,1}$ and $C'_{F^Q,2}$, we find the transfer operator families 
\[
\mc L^{\pm}_{G^Q,s} \sceq 
\begin{pmatrix}
\pm \sum\limits_{n\in\N} \tau_s(Qg_{q-1}^n) & \frac12\tau_s(g_m) \pm \frac12\tau_s(Qg_m) + \sum\limits_{k=m+1}^{q-2} \tau_s(g_k) \pm \tau_s(Qg_k)
\\
\sum\limits_{n\in\N}\tau_s(g_{q-1}^n) \pm \tau_s(Qg_{q-1}^n) & \frac12\tau_s(g_m) \pm \frac12\tau_s(Qg_m) + \sum\limits_{k=m+1}^{q-2} \tau_s(g_k) \pm \tau_s(Qg_k) 
\end{pmatrix}.
\]

By Theorem~\ref{MPtwist}, for $\Rea s > \tfrac12$, the transfer operators $\mc L^\pm_{G^Q,s}$ are nuclear operators of order $0$ on the Banach space $B(\mc D_{q-1})\times B(\mc D_r)$. Moreover, the map $s\mapsto \mc L^\pm_{G^Q,s}$ extends to a meromorphic map on all of $\C$.

\subsection{Zeta functions and Fredholm determinants of transfer operators}
Let 
\[
 [\wt\Gamma_q]_{p,u}\sceq \left\{ [g]_{\wt\Gamma_q} \in [\wt\Gamma_q]_p \left\vert\ [g]_{\wt\Gamma_q} \not= [g_m]_{\wt\Gamma_q},\ [g]_{\wt\Gamma_q} \not= [Qg_m]_{\wt\Gamma_q} \right.\right\}.
\]
We define 
\begin{align*}
Z_-(s) &\sceq \prod_{[g] \in [\wt\Gamma_q]_{p,u}} \prod_{k=0}^\infty \left(1-\det g^{k+1} N(g)^{-(s+k)}\right) \cdot \prod_{\ell=0}^\infty \left(1-N(g_m)^{-(s+2\ell+1)}\right)
\intertext{and}
Z_+(s) &\sceq \prod_{[g]\in [\wt\Gamma_q]_{p,u}} \prod_{k=0}^\infty \left(1-\det g^k N(g)^{-(s+k)}\right)\cdot \prod_{\ell=0}^\infty \left(1-N(g_m)^{-(s+2\ell)}\right).
\end{align*}
These zeta functions converge absolutely for $\Rea s > 1$. Furthermore, we have $Z_+(s)Z_-(s) = Z(s)$ as it is immediately implied by the relation between conjugacy classes of $\Gamma_q$-primitive hyperbolic elements and $\wt\Gamma_q$-hyperbolic primitive elements and, in particular, the fact that, as noted in Section~\ref{sec_strategy}, the $\wt\Gamma_q$-conjugacy classes $[g_m]_{\wt\Gamma_q}$ and $[Qg_m]_{\wt\Gamma_q}$ correspond to the same periodic billiard on $\wt\Gamma_q\backslash\h$. Since we do not rely on this equality in any proofs, we refer to Propositions~\ref{hypeven} and \ref{zetaseven} for more details.

In this section we prove the following relation between the transfer operators $\mc L^\pm_{G^Q,s}$ and the zeta functions $Z_{\pm}$.

\begin{thm}\label{maineven}
For $\Rea s > 1$, we have $\det(1-\mc L^{\pm}_{G^Q,s}) = Z_{\pm}(s)$. Moreover, the zeta functions $Z_\pm$ extend meromorphically to all of $\C$ with possible poles at $s=(1-k)/2$, $k\in\N_0$.
\end{thm}

Let 
\[
 \Gen_{G^Q} \sceq \{ g_m,\ldots, g_{q-2}, Qg_m,\ldots, Qg_{q-2}\} \cup \{g_{q-1}^k, Qg_{q-1}^k \mid k\in\N\}
\]
denote the set of generators of $G^{Q,\eta}$. For $h=h_1\ldots h_n$ with $n\in\N$ and $h_j\in \Gen_{G^Q}$ and $s\in\C$, we define
\[
 b_s^\pm(h) \sceq (\pm 1)^\eps \frac{1}{2^k} \tau_s(h),
\]
where 
\[
\eps = \eps(h) \sceq \#\left\{ j \in \{1,\ldots n\} \left\vert\ h_j \in \{Qg_m,\ldots, Qg_{q-2},Qg_{q-1}^\ell\mid \ell\in\N\} \right.\right\}
\]
and
\[
 k=k(h) \sceq \#\left\{ j \in \{1,\ldots n\} \left\vert\ h_j \in \{g_m,Qg_m\} \right.\right\}.
\]

We define the notions of reduced and regular words over the alphabet $\Gen_{G^Q}$ analogously as before, as well as the sets $W_n^\redu(\Gen_{G^Q})$ for $n\in\N$ and $W_*^\redu(\Gen_{G^Q})$. We remark that, e.g., $g_mQg_m$ and $Qg_mg_m$ are distinct words despite the fact that as elements in $\wt\Gamma_q$ they are identical. Moreover, they are both regular. We also use the notion $W_n^\redu(\{g_m, Qg_m\})$ and $W_*^\redu(\{g_m, Qg_m\})$ for the subset of words whose letters are restricted to $\{g_m, Qg_m\}$.

Analogously as before, for $n\in\N$, we let 
\begin{itemize}
\item $B_1^n$ denote the words in $W_n^\redu(\Gen_{G^Q})$ which end with $g_{q-1}^\ell$ or $Qg_{q-1}^\ell$ for some $\ell\in\N$ and do not begin with $g_{q-1}^k$ for any $k\in\N$,
\item $B_2^n$ denote the words in $W_n^\redu(\Gen_{G^Q})$ which end with an element of \[\{g_k, Qg_k\mid k=m,\ldots, q-2\}\] and do not begin with $g_{q-1}^k$ for any $k\in\N$,
\item $B_3^n$ denote the words in $W_n^\redu(\Gen_{G^Q})$ which end with $g_{q-1}^\ell$ or $Qg_{q-1}^\ell$ for some $\ell\in\N$, and
\item $B_4^n$ denote the words in $W_n^\redu(\Gen_{G^Q})$ which end with an element of \[\{g_k, Qg_k\mid k=m,\ldots, q-2\}.\]
\end{itemize}

A straightforward induction proves the following lemma.

\begin{lemma}
For $n\in\N$ we have
\[
 \left(\mc L^\pm_{G^Q,s}\right)^n =
\begin{pmatrix}
\sum\limits_{a\in B_1^n} b_s^\pm(a) &  \sum\limits_{a\in B_2^n} b_s^\pm(a)
\\
\sum\limits_{a\in B_3^n} b_s^\pm(a) & \sum\limits_{a\in B_4^n} b_s^\pm(a)
\end{pmatrix}.
\]
\end{lemma}

We define
\begin{align*}
[\wt\Gamma_q]_{h,u} & \sceq \big\{ [g^n]_{\wt\Gamma_q} \ \big\vert\  [g]_{\wt\Gamma_q}\in [\wt\Gamma_q]_{p,u},\ n\in\N\big\},
\\
[\wt\Gamma_q]_{p,d} & \sceq \big\{ [g_m]_{\wt\Gamma_q}, [Qg_m]_{\wt\Gamma_q} \big\}, 
\\
[\wt\Gamma_q]_d & \sceq \big\{ [g_m^n]_{\wt\Gamma_q}, [(Qg_m)^n]_{\wt\Gamma_q} \ \big\vert\ n\in\N\big\}, \text{and}
\\
[\wt\Gamma_q]_h & \sceq [\wt\Gamma_q]_{h,u} \cup [\wt\Gamma_q]_d.
\end{align*}

The following proposition is the analog to Proposition~\ref{hyp}. It determines how many regular representatives in $W_*^\redu(\Gen_{G^Q})$ we find for a given element in $[\wt\Gamma_q]_h$.

\begin{prop}\label{hypeven}
\begin{enumerate}[{\rm (i)}]
\item Let $[w]\in [\wt\Gamma_q]_{h,u}$. Then there exists a regular word in the set $W_*^\redu(\Gen_{G^Q})$ which represents $w$. The length of this word is unique, say $\ell(w)$. Let $w_1\ldots w_{\ell(w)}$ be such a representative. Then
\[
 k(w) \sceq \#\left\{ j\in\{1,\ldots,\ell(w)\} \left\vert\ w_j\in\{g_m, Qg_m\} \right.\right\}
\]
does not depend on the choice of the representative. Let $[h]\in [\wt\Gamma_q]_{p,u}$ and $n\in\N$ be the unique elements such that $[h^n] = [w]$. Then there are exactly $2^{k(w)}\ell(h)$ regular representatives of $[w]$ in $W_*^\redu(\Gen_{G^Q})$. These representatives are given as follows:
Let $h_1\ldots h_{\ell(h)}$ be any regular representative of $[h]$ in $W_*^\redu(\Gen_{G^Q})$. Then any of its cyclic permutations $h_j\ldots h_{\ell(h)}h_1\ldots h_{j-1}$ is also a regular representative of $[h]$. This accounts for the factor $\ell(h)$ in the counting. Then 
\[
 w_1\ldots w_{\ell(w)} = (h_j\ldots h_{\ell(h)}h_1\ldots h_{j-1})^n
\]
is a regular representative of $[w]$ in $W_*^\redu(\Gen_{G^Q})$. Suppose that $w_p = g_m$ for some $p\in \{1,\ldots, \ell(w)\}$. If $p\not=\ell(w)$, then
\[
 w_1\ldots w_{p-1} (Qg_m) (Qw_{p+1}) w_{p+2}\ldots w_{\ell(w)}
\]
is also a regular representative of $[w]$. If $p=\ell(w)$, then
\[
 (Qw_1)w_2\ldots w_{\ell(w)-1} (Qg_m)
\]
is also a regular representative. An analogous modification is possible if $w_p=Qg_m$. This accounts for the factor $2^{k(w)}$ in the counting. All regular representatives arise in this way from any chosen first regular representative.
\item If $[w]\in [\wt\Gamma_q]_d$, then $[w]$ is represented by some (regular) word in the set $W_*^\redu(\{g_m, Qg_m\})$. The length of this word is independent of the choice of the representative, say it is $\ell(w)$. If $\det w = 1$ (resp.\@ $\det w=-1$), then $w$ is represented by any word in $W_*^\redu(\{g_m, Qg_m\})$ of length $\ell(w)$ with an even (resp.\@ odd) number of appearances of $Qg_m$. These are all  (regular) representatives of $[w]$ in $W_*^\redu(\Gen_{G^Q})$.
\item For $n\in\N$, the elements in $B_1^n\cup B_4^n$ are precisely the regular representatives in $W_*^\redu(\Gen_{G^Q})$ of all the elements in $[\wt\Gamma_q]_h$ of length $n$.
\end{enumerate}
\end{prop}

\begin{proof}
We define $G^Q$-coding sequences for the elements of $\wh C_{G^Q}$ in the way as explained in Section~\ref{sec_strategy} but using each of the relations $G^{Q,\eta}_{a,b}$, $a,b\in\{1,2\}$, separately. This means that several vectors in $\wh C_{G^Q}$ are assigned multiple $G^Q$-coding sequences. The translation from $G$-coding sequences to $G^Q$-coding sequences works as explained in the proof of Proposition~\ref{hyp} with the difference that whenever the element $g_m$ appears in a $G$-coding sequence, the corresponding element for a $G^Q$-coding sequences can be $g_m$ or $Qg_m$, where in the latter case, the consecutive element is multiplied by $Q$. In other words, if $g_m$ appears in a $G$-coding sequence, then, in the transition to $G^Q$-coding sequences one can choose whether $g_m$ should stay  $g_m$ or being changed to $Qg_mQ$ (and the $Q$ distributed in the correct way to the symbols in the sequence). Now, if the $G$-coding sequence is periodic, then, among the associated $G^Q$-coding sequences, we 
find periodic and non-periodic ones. All of these encode the same periodic billiard. For the proof at hand one has to restrict to those periodic $G^Q$-coding sequences those period length corresponds to the chosen multiplicity of the periodic billiard and hence to the chosen representing $\wt\Gamma_q$-conjugacy class of hyperbolic elements. Taking into account these necessary twists, the proof is analogous to that of Proposition~\ref{hyp}.
\end{proof}

The proof of the following lemma is identical to that of Lemma~\ref{trace}.

\begin{lemma}\label{traceeven}
For $n\in\N$ and $a\in B_1^n\cup B_4^n$ we have
\[
 \Tr b_s^-(a) = \frac{\det a}{2^{k(a)}} \frac{N(a)^{-s}}{1-\det a\cdot N(a)^{-1}}
\]
and
\[
 \Tr b_s^+(a) = \frac{1}{2^{k(a)}} \frac{N(a)^{-s}}{1-\det a\cdot N(a)^{-1}}.
\]
The values of $\Tr b_s^{\pm}(a)$ are invariants for the conjugacy class $[a]_{\wt\Gamma_q}$.
\end{lemma}

\begin{proof}[Proof of Theorem~\ref{maineven}]
We have
\begin{align*}
\log Z_-(s) & = \sum_{[g]\in [\wt\Gamma_q]_{p,u}} \sum_{k=0}^\infty \log\left(1-\det g^{k+1}\cdot N(g)^{-(s+k)}\right)
\\
& \qquad + \sum_{\ell=0}^\infty \log\left(1- N(g_m)^{-(s+2\ell+1)}\right).
\end{align*}
Further,
\begin{align*}
\sum_{[g]\in [\wt\Gamma_q]_{p,u}} \sum_{k=0}^\infty &\log\left(1-\det g^{k+1} \cdot N(g)^{-(s+k)}\right) 
\\
& = -\sum_{\ell=1}^\infty \frac{1}{\ell} \sum_{\stackrel{[w]\in [\wt\Gamma_q]_{h,u}}{\ell(w)=\ell}} \frac{\ell(w)}{n(w)} \cdot\det w\cdot \frac{N(w)^{-s}}{1-\det w\cdot N(w)^{-1}}
\end{align*}
and
\begin{align*}
\sum_{\ell=0}^\infty \log\left(1- N(g_m)^{-(s+2\ell+1)}\right) & = -\sum_{p=1}^\infty \frac1p \cdot\frac{N(g_m^p)^{-(s+1)}}{1-N(g_m^p)^{-2}}.
\end{align*}
Let $[w]\in [\wt\Gamma_q]_{h,u}$, $\ell=\ell(w)$. From Proposition~\ref{hypeven} and Lemma~\ref{traceeven} it follows that
\[
 \sum_{\stackrel{a\in B_1^\ell\cup B_4^\ell}{[a]=[w]}} \Tr b_s^-(a) = 2^{k(w)}\cdot \frac{\ell(w)}{n(w)} \cdot \frac{\det w}{2^{k(w)}} \cdot\frac{N(w)^{-s}}{1-\det w\cdot N(w)^{-1}}.
\]
For any $p\in\N$ and $a\in W^\redu_p(\{g_m, Qg_m\})$ we have
\[
 N(a) = N(a^2)^{1/2} = N(g_m^{2p})^{1/2} = N(g_m^p).
\]
Again Proposition~\ref{hypeven} and Lemma~\ref{traceeven} yield
\begin{align*}
\sum_{\stackrel{a\in B_1^p\cup B_4^p}{a\in W^\redu_*(\{g_m,Qg_m\})}}\Tr b_s^-(a) &= 2^{p-1}\left(\frac{1}{2^p}\cdot \frac{(-1) N(g_m^p)^{-s}}{1+N(g_m^p)^{-1}} + \frac{1}{2^p}\cdot \frac{N(g_m^p)^{-s}}{1-N(g_m^p)^{-1}}\right)
\\
& = \frac{N(g_m^p)^{-(s+1)}}{1-N(g_m^p)^{-2}}.
\end{align*}
Therefore,
\[
 \log Z_-(s) = -\sum_{n=1}^\infty  \frac1n \sum_{a\in B_1^n\cup B_4^n} \Tr b_s^-(a) = -\sum_{n=1}^\infty \frac1n\Tr \left(\mc L_{G^Q,s}^-\right)^n.
\]
Hence, $Z_-(s) = \det(1-\mc L^-_{G^Q,s})$. The meromorphic continuation and the location of possible poles follows from Theorem~\ref{MPtwist}. This completes the proof for $Z_-$. The proof for $Z_+$ is analogous.
\end{proof}

\subsection{Zeta functions and eigenvalues}\label{Venkov_even}

For even $q$, Venkov's Selberg-type zeta function for the Dirichlet boundary value problem is
\begin{align*}
Z^V_-(s) &\sceq \prod_{[g]\in [\Gamma_q]_p} \prod_{k=0}^\infty \left( 1- N(g)^{-(s+k)}\right)^2 \cdot \prod_{\stackrel{[h]\in[\wt\Gamma_q]_{p,u}}{\det h = -1}} \prod_{\ell=0}^\infty \left( \frac{1+N(h)^{-(s+\ell)}}{1-N(h)^{-(s+\ell)}}\right)^{2(-1)^\ell} 
\\
& \quad\times \prod_{n=0}^\infty \left( \frac{1+N(Qg_m)^{-(s+n)}}{1-N(Qg_m)^{-(s+n)}}\right)^{(-1)^n},
\end{align*}
and the one for the Neumann boundary value problem is
\begin{align*}
Z^V_+(s) &\sceq \prod_{[g]\in [\Gamma_q]_p} \prod_{k=0}^\infty \left( 1- N(g)^{-(s+k)}\right)^2 \cdot \prod_{\stackrel{[h]\in[\wt\Gamma_q]_{p,u}}{\det h = -1}} \prod_{\ell=0}^\infty \left( \frac{1-N(h)^{-(s+\ell)}}{1+N(h)^{-(s+\ell)}}\right)^{2(-1)^\ell} 
\\
& \quad\times \prod_{n=0}^\infty \left( \frac{1-N(Qg_m)^{-(s+n)}}{1+N(Qg_m)^{-(s+n)}}\right)^{(-1)^n}.
\end{align*}
As for odd $q$, these zeta functions converge absolutely for $\Rea s>1$ and extend meromorphically to all of $\C$. We denote their meromorphic extensions also by $Z_\pm^V$, respectively. Due to the existence of the boundary periodic geodesic, the relation between the zeta functions $Z_\pm^V$ and $Z_\pm$ is not as strong as for odd $q$. Nevertheless, their relation still implies an analog of Lemma~\ref{zetasident}. To state and prove their relation we define the two zeta functions
\begin{align*}
Z^c_-(s) &\sceq \prod_{k=0}^\infty \left(1-N(g_m^2)^{-(s+k)}\right)^{(-1)^k}
\intertext{and}
Z^c_+(s) &\sceq \prod_{k=0}^\infty \left(1-N(g_m^2)^{-(s+k)}\right)^{(-1)^{k+1}}.
\end{align*}

As before, we denote the meromorphic continuations of $Z_\pm$ also by $Z_\pm$.

\begin{prop}\label{zetaseven}
Let $s\in\C$, $\Rea s>0$. Then $Z^c_{\pm}(s)$ is absolutely convergent and positive. Moreover, $Z^V_{\pm}(s) = Z_{\pm}(s)^4 Z^c_{\pm}(s)$. 
\end{prop}

\begin{proof}
The absolut convergence and positivity of $Z^c_\pm(s)$ follows from $N(g_m^2)>1$. The relation between the elements of $[\Gamma_q]_p$ and $[\wt\Gamma_q]_{p,u} \cup [\wt\Gamma_q]_{p,d}$ is as follows:
\begin{enumerate}[1)]
\item If $[g]_{\Gamma_q} \in [\Gamma_q]_p$, $[g]_{\Gamma_q}\not=[g_m]_{\Gamma_q}$ and $g$ is not $\wt\Gamma_q$-primitive, then there exists  $[h]_{\wt\Gamma_q} \in [\wt\Gamma_q]_{p,u}$ with $\det h = -1$ and $[h^2]_{\Gamma_q} = [g]_{\Gamma_q}$.
\item If $[g]_{\Gamma_q} \in [\Gamma_q]_p$, $[g]_{\Gamma_q}\not=[g_m]_{\Gamma_q}$ and $g$ is $\wt\Gamma_q$-primitive, then $QgQ$ is also $\wt\Gamma_q$-primitive and $[g]_{\Gamma_q} \not= [QgQ]_{\Gamma_q}$, but $[g]_{\wt\Gamma_q} = [QgQ]_{\wt\Gamma_q}$.
\item The element $g_m$ is $\Gamma_q$-primitive and $\wt\Gamma_q$-primitive hyperbolic, and, since $g_m=Qg_mQ$, we have $[g_m]_{\Gamma_q} = [Qg_mQ]_{\Gamma_q}$ and $[g_m]_{\wt\Gamma_q} = [Qg_mQ]_{\wt\Gamma_q}$.
\item The element $Qg_m$ is $\wt\Gamma_q$-primitive with $\det Qg_m = -1$, but $(Qg_m)^2 = g_m^2$ is not $\Gamma_q$-primitive.
\item If $[h]_{\wt\Gamma_q} \in [\wt\Gamma_q]_{p,u}$ with $\det h=-1$, then $[h^2]_{\Gamma_q} \in [\Gamma_q]_p$.
\item If $[h]_{\wt\Gamma_q} \in [\wt\Gamma_q]_{p,u}$ with $\det h = 1$, then $[h]_{\wt\Gamma_q}\cap \Gamma_q$ splits into the two distinct classes $[h]_{\Gamma_q}, [QhQ]_{\Gamma_q}\in [\Gamma_q]_p$.
\end{enumerate}
Therefore we have
\begin{align*}
Z^V_-(s) &  = \prod_{\stackrel{[a]\in [\wt\Gamma_q]_{p,u}}{\det a = 1}} \prod_{k=0}^\infty \left(1-N(a)^{-(s+k)}\right)^4 \cdot \prod_{\stackrel{[b]\in [\wt\Gamma_q]_{p,u}}{\det b=-1}} \prod_{\nu=0}^\infty \left(1-N(b^2)^{-(s+\nu)}\right)^2
\\
& \quad\times \prod_{\mu=0}^\infty \left(1-N(g_m)^{-(s+\mu)}\right)^2 \cdot \prod_{\stackrel{[h]\in [\wt\Gamma_q]_{p,u}}{\det h  =-1}} \prod_{\ell=0}^\infty \left( \frac{1+N(h)^{-(s+\ell)}}{1-N(h)^{-(s+\ell)}}\right)^{2(-1)^\ell} 
\\
& \quad\times \prod_{n=0}^\infty \left( \frac{1+N(g_m)^{-(s+n)}}{1-N(g_m)^{-(s+n)}}\right)^{(-1)^n}
\\
& = \prod_{[g]\in [\wt\Gamma_q]_{p,u}} \prod_{k=0}^\infty \left(1-\det g^{k+1}\cdot N(g)^{-(s+k)}\right)^4 
\\
& \quad\times \prod_{n=0}^\infty \left(1-N(g_m)^{-(s+n)}\right)^2 \left( \frac{1+N(g_m)^{-(s+n)}}{1-N(g_m)^{-(s+n)}}\right)^{(-1)^n}
\\
& = Z_-(s)^4 Z^c_-(s).
\end{align*}
This completes the proof for $Z_-$. The consideration of $Z_+$ is analogous.
\end{proof}

\begin{proof}[Proof of Theorem~\ref{mainintro} for even $q$]
The statements follow from a direct composition of Theorem~\ref{Venkov}, Proposition~\ref{zetaseven} and Theorem~\ref{maineven}.
\end{proof}


\section{Concluding remarks}\label{conclusion}

Theorem~\ref{mainintro}, in connection with Theorem~\ref{Venkov}, allow to investigate the existence of odd and even Maass cusp forms via investigating the spectrum of the operators $\mc L^\pm_{G^Q,s}$. We postpone investigations of this kind to future work. In particular, the Phillips-Sarnak conjecture on the nonexistence of even Maass cusp forms for nonarithmetic Hecke triangle groups can be formulated equivalently as that for $\Rea s = \tfrac12$, the transfer operators $\mc L^+_{G^Q,s}$ (rather their meromorphic continuations in $s$) for the Hecke triangle groups $\Gamma_q$, $q\notin \{3,4,6\}$, do not have $1$-eigenfunctions in the Banach space $B(\mc D_{q-1})\times B(\mc D_r)$. 

Despite that, the conducted investigations have two by-products. From Lemma~\ref{zetasident} and Theorem~\ref{mainodd} respectively from Proposition~\ref{zetaseven} and Theorem~\ref{maineven} and the localization of poles for $Z^V_-$ in \cite{Venkov_book} it follows that the Fredholm determinants $\det(1-\mc L_{G^Q,s}^-)$ do not have a pole at $s=\tfrac12$. Moreover, for Hecke triangle groups $\Gamma_q$ with even $q$, the factorization of the Selberg zeta function $Z = Z_-\cdot Z_+$ is not identical to Venkov's factorization $Z^4 = Z^V_- \cdot Z^V_+$ (see Proposition~\ref{zetaseven}). The auxiliary zeta function $Z^c_-$, which links $Z_-$ with $Z^V_-$ and $Z_+$ with $Z^V_+$, has a zero for $s=0$. Therefore, while $Z_-$ and $Z^V_-$ as well as $Z_+$ and $Z^V_+$ have the same zeros and poles on $\Rea s>0$, this need not be true on $\Rea s\leq 0$. It is still an open question whether this transfer operator induced factorization has a spectral explanation.


\subsubsection*{Acknowledgement} The author would like to thank the referee for a thorough reading and helpful comments.

%

\providecommand{\bysame}{\leavevmode\hbox to3em{\hrulefill}\thinspace}
\providecommand{\MR}{\relax\ifhmode\unskip\space\fi MR }
\providecommand{\MRhref}[2]{%
  \href{http://www.ams.org/mathscinet-getitem?mr=#1}{#2}
}
\providecommand{\href}[2]{#2}

\end{document}